\documentclass[twoside,10pt]{article}
\usepackage[colorlinks]{hyperref}                        % book-mark
\usepackage{amsmath,latexsym,amsfonts,indentfirst,amsthm,amsxtra,amssymb,bm}

\usepackage[perpage,symbol*]{footmisc}
\usepackage[numbers,sort&compress]{natbib}

\numberwithin{equation}{section}

 \newtheorem{lemma}{Lemma}[section]
 \newtheorem{proposition}[lemma]{Proposition}
 \newtheorem{theorem}{Theorem}
 \newtheorem{corollary}[lemma]{Corollary}

 \theoremstyle{remark}
 \newtheorem{remark}{Remark}[section]
\begin{document}

%--------------------------------------------------------------------------------------
%--------------------------------------------------------------------------------------
\title{\bf Asymptotic stability of wave patterns
to compressible viscous and heat-conducting gases
in the half space}
\author{{\bf Ling Wan}\footnote{
School of Mathematics and Statistics, Wuhan University,
Wuhan 430072, China
(ling.wan@whu.edu.cn).
This author was partially supported by the Fundamental Research Funds for the Central Universities under contract 2015201020203.}\,\,,
{\bf Tao Wang}\footnote{
School of Mathematics and Statistics, Wuhan University,
Wuhan 430072, China
(tao.wang@whu.edu.cn).}\,\,,
{\bf Huijiang Zhao}\footnote{
School of Mathematics and Statistics, Wuhan University,
Wuhan 430072, China (hhjjzhao@whu.edu.cn).
This author was partially supported by the grants
from the National Natural Science Foundation of China under contracts
10925103, 11271160, and 11261160485.
}}

\date{}

\maketitle
\vspace*{-5mm}
% \thispagestyle{style}                             % 当前页的页面式样, style=empty 当前页不显示页眉页脚, style=first

%--------------------------------------------------------------------------------------
%abstract------------------------------------------------------------------------------
\begin{abstract}
We study the large-time behavior of solutions
to the compressible Navier-Stokes equations
for a viscous and heat-conducting ideal polytropic gas
in the one-dimensional half-space.
A rarefaction wave
and its superposition with a non-degenerate stationary solution
are shown to be asymptotically stable for the outflow problem
with large initial perturbation and general adiabatic exponent.

  \bigbreak
  \bigbreak
\end{abstract}

%\tableofcontents
\section{Introduction}
The one-dimensional motion of a compressible viscous and heat-conducting gas in the
half space $\mathbb{R}_+:=(0,\infty)$ can be formulated by the compressible
Navier-Stokes equations
\begin{equation} \label{NS_E}
  \left\{
  \begin{aligned}
    \rho_t+(\rho u)_x&=0,\\
    (\rho u)_t+(\rho u^2+ P)_x&=(\mu u_x)_x,\\
    (\rho{E})_t+(\rho u{E}+uP)_x&
    =(\kappa\theta_x+\mu uu_x)_x,
  \end{aligned}
  \right.
\end{equation}
where $t>0$ and $x\in\mathbb{R}_+$ stand for the time variable and
the spatial variable, respectively,
and the primary dependent variables are the density $\rho$, the velocity
$u$ and the temperature $\theta$. The specific total energy
${E}=e+\frac12u^2$
with $e$ being the specific internal energy.
It is known from thermodynamics that only two of the thermodynamic
variables $\rho,$ $\theta$, $P$ (pressure), $e$ and $s$ (specific entropy) are
independent.
We focus on the ideal polytropic gas, which is expressed in normalized units
by the following constitutive relations
\begin{equation}\label{ideal}
  P=R\rho\theta,   \quad e=c_v\theta, \quad
  s=c_v\ln(\rho^{1-\gamma}\theta),
\end{equation}
where $R>0$ is the gas constant, $\gamma>1$ the adiabatic exponent and
 $c_v=R/(\gamma-1)$ the specific heat at constant volume.
Positive constants $\mu$ and $\kappa$ are the viscosity
and the heat conductivity, respectively.

The system \eqref{NS_E}-\eqref{ideal} is supplemented with the initial condition
\begin{equation}\label{initial}
(\rho, u, \theta)|_{t=0}=(\rho_0, u_0, \theta_0),
\end{equation}
which is assumed to satisfy the far-field condition
\begin{equation}\label{infinity}
  \lim_{x\to\infty}(\rho_0, u_0, \theta_0)(x)=(\rho_+,u_+,\theta_+),
\end{equation}
where $\rho_+>0$, $u_{+}$ and $\theta_{+}>0$ are constants.
For boundary conditions, we take
\begin{equation}
  \label{bdy}
   (u,\theta)(t,0)=(u_-,\theta_{-}),
\end{equation}
where $u_{-}$ and $\theta_->0$ are constants.
The initial data \eqref{initial} is assumed to satisfy
certain compatibility conditions as usual.

The boundary condition $u(t,0)=u_-<0$ means that
the fluid blows out from the boundary,
and hence the initial boundary value problem \eqref{NS_E}-\eqref{bdy} with
$u_-<0$ is called the outflow problem.
The problem \eqref{NS_E}-\eqref{bdy} with $u_-=0$ is called
the impermeable wall problem, which has been studied in
\cite{HM09MR2511653,HLSMR2730324,WZZ13MR3063535,MN00MR1738558,
MM99MR1682659} and so on.
According to the theory of well-posedness for initial boundary value problem,
one has to impose one extra boundary condition
$\rho(t,0)=\rho_-$ on $\{x=0\}$ for the case when
$u_->0$.
This case is called the inflow problem and has been
investigated by Matsumura et al.
\cite{QW11MR2765694,QW09MR2578799,HLSMR2730324,
FLWZMR3260233,HMS03MR1997442,MN01MR1888084}.
We refer to Matsumura \cite{Ma01MR1944189}
for a complete classification about
the large-time behaviors of solutions
to initial boundary value problems of the isentropic compressible
Navier-Stokes equations in the half space $\mathbb{R}_+$.

The main purpose of this article is to study the large-time behavior of
solutions to the outflow problem \eqref{NS_E}-\eqref{bdy}.
The nonlinear stability
of the stationary solution, the rarefaction wave and their composition
has been addressed in \cite{KNNZMR2755498,Q11MR2785974} under
small initial perturbation.
For large perturbation case,
Qin  \cite{Q11MR2785974} proved that
the non-degenerate stationary solution
is asymptotically stable under the technical assumption
that the adiabatic exponent $\gamma$ is close to 1.
Recently, Wan et al. \cite{WWZ15} establish the asymptotic stability of
the non-degenerate stationary solution
with large initial perturbation and general adiabatic exponent $\gamma$.
In this article we are going to study the case when the corresponding
time-asymptotic state is a rarefaction wave or its superposition with
a non-degenerate stationary solution under large initial perturbation.

We first investigate the large-time behavior of solutions toward
the rarefaction wave for
the outflow problem or the impermeable wall problem
\eqref{NS_E}-\eqref{bdy}.
To this end,
we assume that positive constants $\rho_+$, $u_{\pm}$ and $\theta_{\pm}$
satisfy
\begin{equation}\label{Thm1a}
  u_-=u_++\int_{\rho_+}^{(\theta_-/\theta_+)^{\frac{1}{\gamma-1}}\rho_+}
  \sqrt{R\gamma\rho_+^{1-\gamma}\theta_+}z^{\frac{\gamma-3}{2}}
  \mathrm{d}z< u_+,
\end{equation}
so that $(\rho_+,u_+,\theta_+)\in R_3(\rho_-,u_-,\theta_-)$ for
\begin{equation}\label{rho-}
  \rho_-=(\theta_-/\theta_+)^{\frac{1}{\gamma-1}}\rho_+.
\end{equation}
Here $R_3(\rho_-,u_-,\theta_-)$ is the $3$-rarefaction wave curve through
$(\rho_-,u_-,\theta_-)$ given by
\begin{equation*}
  R_3(\rho_-,u_-,\theta_-)
  :=\left\{
  (\rho,u,\theta)\left|
  \begin{aligned}
     &\rho>\rho_-,\ \rho^{1-\gamma}\theta=\rho_-^{1-\gamma}\theta_-,\\[-1mm]
  &u=u_-+\int_{\rho_-}^{\rho}
  \sqrt{R\gamma\rho_-^{1-\gamma}\theta_- }z^{\frac{\gamma-3}{2}}\mathrm{d}z
  \end{aligned}
  \right.
\right\}.
\end{equation*}
We assume further that
\begin{equation}\label{Thm1b}
u_-+\sqrt{R\gamma\theta_-}\geq 0.
\end{equation}
Then as time $t$ goes to infinity,
the solution of the problem \eqref{NS_E}-\eqref{bdy} is expected to
converge to the $3$-rarefaction wave $(\rho^R,u^R,\theta^R)(t,x)$
connecting $(\rho_{-},u_{-},\theta_{-})$ and $(\rho_{+},u_{+},\theta_{+})$,
which is the unique entropy solution of the Riemann problem for
the corresponding hyperbolic system of \eqref{NS_E}-\eqref{ideal}
(i.e. the compressible Euler system)
\begin{equation} \label{Euler}
  \left\{
  \begin{aligned}
    \rho_t+(\rho u)_x&=0,\\
    (\rho u)_t+(\rho u^2+ P)_x&=0,\\
    (\rho{E})_t+(\rho u{E}+uP)_x&=0
  \end{aligned}
  \right.
\end{equation}
for $(t,x)\in\mathbb{R}_+\times\mathbb{R}$ with initial data
\begin{equation} \label{Riemann}
  (\rho,u,\theta)(0,x)=
  \left\{
  \begin{aligned}
    &(\rho_-,u_-,\theta_-)& & \quad{\rm for}\ x<0, \\
    &(\rho_+,u_+,\theta_+)& & \quad{\rm for}\ x>0.
  \end{aligned}
  \right.
\end{equation}
%%%%%% cf.的意思是
%%%%%% used in writing when you want the reader to make a comparison
%%%%%% between the subject being discussed and something else
Note that
$(\rho^R,u^R,\theta^R)(t,x)\equiv (\rho_{-},u_{-},\theta_{-})$
for each $(t,x)\in[0,\infty)\times(-\infty,0]$ due to \eqref{Thm1b}.

We  construct a smooth approximation
$(\bar{\rho},\bar{u},\bar{\theta})$ of
$(\rho^R,u^R,\theta^R)$ for deriving the
stability of the rarefaction wave.
As in \cite{HMS03MR1997442}, we consider the Cauchy problem for
the Burgers equation
%%\begin{equation}\label{Burgers}
%%  \left\{
%%  \begin{aligned}
%%    w_t+ww_x&=0\qquad {\rm for}\ (t,x)\in\mathbb{R}_+\times\mathbb{R},\\
%%    w(0,x)&=
%%    \left\{
%%    \begin{array}{rlr}
%%      &w_-\quad &{\rm for}\ x<0,\\
%%      &w_-+ \tilde{w}k_q\int_0^x z^q{\rm e}^{-z}\mathrm{d}z \quad &{\rm for}\ x\geq 0,
%%    \end{array}
%%    \right.
%%  \end{aligned}
%%  \right.
%%\end{equation}
\begin{equation}\label{Burgers}
  \left\{
  \begin{aligned}
    w_t+ww_x&=0&&\quad {\rm for}\ (t,x)\in\mathbb{R}_+\times\mathbb{R},\\[-0.5mm]
    w(0,x)&=
      w_-+ \tilde{w}k_q\int_0^{x_+} z^q{\rm e}^{-z}\mathrm{d}z &&\quad {\rm for}\ x\in\mathbb{R},
  \end{aligned}
  \right.
\end{equation}
where $\tilde{w}:=w_+-w_-$, $q\geq 16$ is some fixed constant,
$x_+:=\max\{x,0\}$ is the positive part of $x$, and
 the constant $k_q$ satisfies
$$k_q\int_0^{\infty} z^q{\rm e}^{-z}\mathrm{d}z=1.$$
The smoothed rarefaction wave  $(\bar{\rho},\bar{u},\bar{\theta})$
connecting $(\rho_{-},u_{-},\theta_{-})$ and
 $(\rho_{+},u_{+},\theta_{+})$  is defined by
\begin{equation}\label{SRW}
  \left\{
  \begin{aligned}
    \lambda_3(\bar{\rho},\bar{u},\bar{\theta})(t,x)&=w(1+t,x),\\[2mm]
    (\bar{\rho}^{1-\gamma}\bar{\theta})(t,x)&=\rho_+^{1-\gamma}\theta_+,\\
  \bar{u}(t,x)&=u_++\int_{\rho_+}^{\bar{\rho}(t,x)}
  \sqrt{R\gamma\rho_+^{1-\gamma}\theta_+ }z^{\frac{\gamma-3}{2}}\mathrm{d}z,
  \end{aligned}\right.
\end{equation}
where
$\rho_-$ is given by \eqref{rho-} and
$w(t,x)$ is the unique solution of \eqref{Burgers} with
$w_{\pm}=\lambda_3(\rho_{\pm},u_{\pm},\theta_{\pm}).$

Now we state our stability result of the
rarefaction wave $(\rho^R,u^R,\theta^R)(t,x)$ to the outflow problem or
the impermeable wall problem
with large initial perturbation.
\begin{theorem} \label{thm1}
  Assume that $(\rho_+,u_{\pm},\theta_{\pm})$
   and the initial data $(\rho_0,u_0,\theta_0)$
  satisfy \eqref{Thm1a}, \eqref{Thm1b} and
  \begin{equation}\label{Thm1c}
    \inf_{x\in\mathbb{R}_+}\left\{\rho_0(x),\theta_0(x)\right\}>0,\quad
    (\rho_0-\bar{\rho},u_0-\bar{u},\theta_0-\bar{\theta})\in H^1(\mathbb{R}_+).
  \end{equation}
  Then there exists a positive constant $\epsilon_1$ such that
  if $u_-\leq 0$ and the boundary strength
  $\delta:=\left|(u_+-u_-,\theta_+-\theta_-)\right|\leq \epsilon_1$,
  the initial boundary value problem
  \eqref{NS_E}-\eqref{bdy}
  admits a unique solution
  $(\rho,u,\theta)$ satisfying
  \begin{equation} \label{s_o_s1}  %%%%%%%%%% space of solution
      \begin{gathered}
    (\rho-\bar{\rho},u-\bar{u},\theta-\bar{\theta}) \in C([0,\infty);H^1(\mathbb{R}_+)),\\
    \rho_x-\bar{\rho}_x\in L^2(0,\infty;L^2(\mathbb{R}_+)),\quad
    (u_x-\bar{u}_x,\theta_x-\bar{\theta}_x)\in L^2(0,\infty;H^1(\mathbb{R}_+)).
    \end{gathered}
  \end{equation}
  Furthermore, the solution $(\rho,u,\theta)$
  converges to the rarefaction wave $(\rho^R,u^R,\theta^R)$
  uniformly as time tends to infinity:
  \begin{equation}\label{stability1}
    \lim_{t\to\infty}\sup_{x\in\mathbb{R}_+}
    \left|(\rho-{\rho}^R,u-{u}^R,\theta-{\theta}^R)(t,x)\right|=0.
  \end{equation}
\end{theorem}
\begin{remark}
For the Cauchy problem to the compressible
  Navier-Stokes equations \eqref{NS_E}-\eqref{ideal}
  with generic adiabatic exponent $\gamma$ in the whole space $\mathbb{R}$,
   we can employ the methodology developed in this paper to
   obtain the time-asymptotic stability of the rarefaction waves
    under large initial perturbation,
   which extends the corresponding stability results
   in \cite{KMN86MR868811,NYZMR2083790} for the case with
   small initial perturbation or
    the case when $\gamma$ is close to $1$.
    We refer to a recent work \cite{HW15} for the
    stability of superposition of viscous contact wave
    and rarefaction waves to the Cauchy problem for the
    compressible Navier-Stokes equations in Lagrangian coordinate.
   \end{remark}

Next we intend to study the time-asymptotic stability of
the superposition of a non-degenerate stationary solution and a $3$-rarefaction wave.
For this purpose, we let $(\tilde{\rho},\tilde{u},\tilde{\theta})$
be the stationary solution of \eqref{NS_E}-\eqref{bdy} connecting
$({u}_-,{\theta}_-)$ and $({\rho}_m,{u}_m,{\theta}_m)$, namely
$(\tilde{\rho},\tilde{u},\tilde{\theta})$
depends solely on the variable $x$ and satisfies
\begin{equation}\label{BL1}
 \left\{
  \begin{aligned}
    (\tilde{\rho}\tilde{u})'&=0,\\
    (\tilde{\rho}\tilde{u}^2+ \tilde{P})'&=\mu \tilde{u}'',\\[0.9mm]
    (\tilde{\rho} \tilde{u}\tilde{E}
    +\tilde{u}\tilde{P})'
    &=\kappa\tilde{\theta}''+\mu (\tilde{u}\tilde{u}')'
  \end{aligned}
  \right.
\end{equation}
for $x\in\mathbb{R}_+$ and
\begin{equation}\label{BL2}
      (\tilde{u},\tilde{\theta})(0)=(u_-,\theta_-),\quad
    \lim_{x\to\infty}(\tilde{\rho},\tilde{u},\tilde{\theta})(x)
    =(\rho_m,u_m,\theta_m).
\end{equation}
where
$\tilde{P}:=R\tilde{\rho}\tilde{\theta}$
and
$\tilde{E}:=c_v\tilde{\theta}+\frac12\tilde{u}^2$.
A stationary solution $(\tilde{\rho},\tilde{u},\tilde{\theta})$ is
called to be
non-degenerate if for each $n\in\mathbb{N}$,
  \begin{equation}\label{decay1}
   |\partial_x^n(\tilde{\rho}-\rho_m,\tilde{u}-u_m,\tilde{\theta}-\theta_m)(x)|
   \leq C\tilde\delta \mathrm{e}^{-cx},
  \end{equation}
where
$C$, $c$ are positive constants and
$\tilde\delta:=\left|(u_m-u_-,\theta_m-\theta_-)\right|$
is the boundary strength of the stationary solution.
The existence of (non-degenerate) stationary solutions has been
shown by Kawashima et al. \cite{KNNZMR2755498} and
will be restated in section \ref{sec_pre}.

We assume that
\begin{equation}\label{Thm2a}
  (\rho_+,u_+,\theta_+)\in R_3(\rho_m,u_m,\theta_m),\quad
  \sqrt{R\gamma\theta_m}\geq -u_m>0,
\end{equation}
so that there exist a $3$-rarefaction wave
 $(\rho^R,u^R,\theta^R)$
connecting $(\rho_{m},u_{m},\theta_{m})$ and $(\rho_{+},u_{+},\theta_{+})$.
It is expected that the large-time behavior of solutions to the outflow problem
\eqref{NS_E}-\eqref{bdy} is determined by
the composition  $(\check\rho,\check{u},\check\theta)$
of the stationary solution $(\tilde{\rho},\tilde{u},\tilde{\theta})$
and the $3$-rarefaction wave  $(\rho^R,u^R,\theta^R)$:
\begin{equation}\label{check}
  (\check\rho,\check{u},\check\theta)(t,x)
  =(\tilde{\rho},\tilde{u},\tilde{\theta})(x)+
  (\rho^R,u^R,\theta^R)(t,x)-(\rho_{m},u_{m},\theta_{m}).
\end{equation}

In order to derive the stability result, we introduce
the smoothed asymptotic state
\begin{equation}\label{hat}
  (\hat\rho,\hat{u},\hat\theta)(t,x)
  =(\tilde{\rho},\tilde{u},\tilde{\theta})(x)+
  (\bar{\rho},\bar{u},\bar{\theta})(t,x)-(\rho_{m},u_{m},\theta_{m}),
\end{equation}
where $(\bar{\rho},\bar{u},\bar{\theta})$ is the
 smooth  rarefaction wave
connecting $(\rho_{m},u_{m},\theta_{m})$ and $(\rho_{+},u_{+},\theta_{+})$.
We define
$$\bar\delta:=\left|(u_m-u_+,\theta_m-\theta_+)\right|.
$$

Under the above preparation, we have the following stability result
on the superposition  $(\check\rho,\check{u},\check\theta)$
with large initial perturbation.
\begin{theorem}
  \label{thm2}
  Assume that there exists a non-degenerate stationary solution
  $(\tilde{\rho},\tilde{u},\tilde{\theta})$ connecting
$({u}_-,{\theta}_-)$ and $({\rho}_m,{u}_m,{\theta}_m)$.
Assume further that $(\rho_+,u_{+},\theta_{+})$
   and the initial data $(\rho_0,u_0,\theta_0)$
  satisfy \eqref{Thm2a}  and
  \begin{equation}\label{Thm2b}
    \inf_{x\in\mathbb{R}_+}\left\{\rho_0(x),\theta_0(x)\right\}>0,\quad
    (\rho_0-\hat{\rho},u_0-\hat{u},\theta_0-\hat{\theta})\in H^1(\mathbb{R}_+).
  \end{equation}
Then a positive constant  $\epsilon_2$ exists such that
 the outflow problem $\eqref{NS_E}$-\eqref{bdy}
 with  $\bar\delta+\tilde\delta\leq \epsilon_2$ has a unique solution
  $(\rho,u,\theta)$ satisfying
  \begin{equation} \label{s_o_s2}  %%%%%%%%%% space of solution
      \begin{gathered}
    (\rho-\hat{\rho},u-\hat{u},\theta-\hat{\theta}) \in C([0,\infty);H^1(\mathbb{R}_+)),\\
    \rho_x-\hat{\rho}_x\in L^2(0,\infty;L^2(\mathbb{R}_+)),\quad
    (u_x-\hat{u}_x,\theta_x-\hat{\theta}_x)\in L^2(0,\infty;H^1(\mathbb{R}_+)).
    \end{gathered}
  \end{equation}
  Furthermore, the solution $(\rho,u,\theta)$
  converges to the composition $(\check\rho,\check{u},\check\theta)$ of
  the stationary solution $(\tilde{\rho},\tilde{u},\tilde{\theta})$
and the rarefaction wave  $(\rho^R,u^R,\theta^R)$
  uniformly as time tends to infinity:
  \begin{equation}\label{stability2}
    \lim_{t\to\infty}\sup_{x\in\mathbb{R}_+}
    |(\rho-\check\rho,u-\check{u},\theta-\check{\theta})(t,x)|=0.
  \end{equation}
\end{theorem}

To derive the large-time behavior of solutions
to the compressible Navier-Stokes equations \eqref{NS_E},
it suffices to deduce certain uniformly-in-time
 \emph{a priori} estimates on the perturbations
 toward the asymptotic state
 and the essential step
is to obtain the positive lower and upper  bounds on
 the density $\rho(t,x)$ and the temperature $\theta(t,x)$
uniformly in time $t$ and space $x$.
In the case of small perturbation,
one can use the smallness of
the \emph{a priori} $H^1$-norm of the perturbation to
get the uniform bounds of the density $\rho$ and the temperature $\theta$.
Owing to such uniform bounds and the smallness of the boundary strength $\delta$,
one can derive certain uniform \textmd{a priori} energy-type estimates as shown in \cite{Q11MR2785974,KNNZMR2755498}.
In the case that the adiabatic exponent $\gamma$ is close to 1,
by observing that $
\theta=\rho^{\gamma-1}\mathrm{e}^{(\gamma-1)s/R}
$
for ideal polytropic gases \eqref{NS_E}-\eqref{ideal},
one can deduce that
$\|\theta-1\|_{L^\infty([0,T]\times\mathbb{R})}$
can be sufficiently small.
Thus the desired energy-type
a priori estimates can be performed
as in \cite{NYZMR2083790,Q11MR2785974}
based on the smallness of $\delta$ and the a priori assumption
\begin{equation*}
\tfrac 12\leq\theta(t,x)\leq 2\quad {\rm for\ all}\
(t,x)\in[0,T]\times\mathbb{R}.
\end{equation*}

However, these arguments are no longer valid for
the case with  large initial perturbation
and general adiabatic exponent.
We note that,
even for the asymptotic stability of constant state to
the Cauchy problem for the system \eqref{NS_E},
the uniform positive
lower and upper bounds on $\theta(t,x)$ are given
only very recently by  Li and Liang \cite{LL},
although the corresponding uniform  bounds on $\rho(t,x)$
were addressed in \cite{J02MR1912419,J99MR1671920} thirteen years ago.
In their work \cite{LL},
Li-Liang considered
the fixed-domain problems to the compressible Navier-Stokes equations
in the Lagrangian coordinate
and obtained the uniform positive lower and upper bounds
on the temperature $\theta(t,x)$
through a time-asymptotically nonlinear stability analysis.
However, the outflow problem \eqref{NS_E}-\eqref{bdy}
will be transformed into a free boundary problem in the Lagrangian coordinate,
which makes the treatment of boundary more difficult.
To overcome this difficulty, we shall make use of a direct
energy method to to the reformulated problem
for the compressible Navier-Stokes equations \eqref{NS_E} in the Eulerian
coordinate and take account of the dissipative effect of the boundary terms.

The main point for deriving our main results, the stability of the
rarefaction wave and its superposition
with a non-degenerate stationary solution to the
initial boundary value problem \eqref{NS_E}-\eqref{bdy},
is to employ the smallness of the boundary strength $\delta$ to
control the possible growth of the perturbation suitably.
Specifically,
we first deduce the basic energy estimate
with the aid of the decay properties
of the smoothed rarefaction wave and the non-degenerate stationary solution
 provided that
the boundary strength $\delta$ multiplied with a certain function of
$m_1$ (the \emph{a priori} lower bound of density $\rho$), $m_2$ (the
\emph{a priori} lower bound of temperature $\theta$) and $N$
(the \emph{a priori}  bound of  the $L^{\infty}(0,T;H^1(\mathbb{R}_+))$-norm
of  perturbation)
is suitably small (see Lemma \ref{lem_bas} for detailed statement).
Next, to get uniform pointwise bounds of the density $\rho(t,x)$,
 we transform the outflow problem \eqref{NS_E}-\eqref{bdy} into
a free boundary problem in the Lagrangian coordinate and
 modify Jiang's argument for fixed domains in \cite{J99MR1671920,J02MR1912419}.
Especially, we will use a cut-off function with parameter to
localize the free boundary problem, and then
we will deduce a local representation of the specific volume $v=1/\rho$ to
establish the uniform bounds of $v$.
With such uniform bounds of the density $\rho$ in hand,
we can derive the $H^1$-norm (in the spatial variable $x$) estimate of
the perturbation uniformly in the time $t$ in the Eulerian coordinate.
And the maximum principle enables us to get the
positive lower bound of the temperature
$\theta(t,x)$ locally in time $t$.
In view of the \emph{a priori} assumption \eqref{asmp1},
we have to obtain the uniform positive lower bound of the temperature $\theta(t,x)$,
which will be achieved
by combining the locally-in-time lower bound
of $\theta(t,x)$ and a well-designed continuation argument.

The layout of this paper is as follows.
After stating the notations, we summarize the
existence of the stationary solution and some
properties of the smoothed rarefaction wave in Section \ref{sec_pre}.
The basic energy estimate, the uniform bounds of the density $\rho$,
the uniform  $H^1$-norm estimate
 and the  locally-in-time lower bound of the temperature $\theta$
will be obtained in subsections \ref{sec_bas}, \ref{sec_den},
\ref{sec_norm} and \ref{sec_tem}, respectively.
The last part of this manuscript, subsection \ref{sec_proof},
 is devoted to showing the proof of our main results
by applying a well-designed continuation argument.

\bigbreak
\noindent \emph{Notations.}
Throughout this paper,
$L^q(\mathbb{R}_+)$ $(1\leq q\leq \infty)$ stands for
the usual Lebesgue space on $\mathbb{R}_+$  equipped with the norm $\|{\cdot}\|_{L^q}$
and $H^k(\mathbb{R}_+)$  $(k\in \mathbb{N})$
the usual Sobolev space in the $L^2$ sense
with norm $\|\cdot\|_k$.
We introduce $\|\cdot\|=\|\cdot\|_{L^2(\mathbb{R}_+)}$ for simplicity.
The space of  continuous
 functions on the interval $I$ with values in
$H^k(\mathbb{R}_+)$ is denoted by $C(I; H^k(\mathbb{R}_+))$
or simply by $C(I; H^k)$
while the space of $L^2$-functions
on $I$ with values in $H^k(\mathbb{R}_+)$
is denoted by $L^2(I; H^k(\mathbb{R}_+))$ or simply by
$L^2(I; H^k )$.
The Gaussian bracket $[x]$ means
the largest integer not greater than $x$,
and $x_+:=\max\{x,0\}$ is the positive part of $x$.

\section{Preliminaries}\label{sec_pre}
It is well-known (see \cite{S94MR1301779,Da10MR2574377} for example) that
for each $(t,x)\in \mathbb{R}_+\times\mathbb{R}$,
\begin{equation*}
  (\rho^R,u^R,\theta^R)(t,x)\in R_3(\rho_-,u_-,\theta_-),\quad
  \lambda_3(\rho^R,u^R,\theta^R)(t,x)=w^{R}(t,x),
\end{equation*}
where
\begin{equation*}%\label{speed}
  \lambda_3(\rho ,u,\theta):=u+\sqrt{R\gamma\theta}
\end{equation*}
is the $3$-characteristic speed of the system \eqref{Euler} and
$w^R(t,x)$ is the continuous weak solution of the Riemann problem on
Burgers equation
\begin{equation*} %\label{Burgers2}
  \left\{
  \begin{aligned}
    w^R_t+w^Rw^R_x&=0&& {\rm for}\ (t,x)\in\mathbb{R}_+\times\mathbb{R},\\
    w^R(0,x)&=w_{\pm} &&{\rm for}\ \pm x>0
  \end{aligned}
  \right.
\end{equation*}
with $w_{\pm}=\lambda_3(\rho_{\pm},u_{\pm},\theta_{\pm})$.
Moreover, $w^R(t,x)$ takes the form of
\begin{equation*}
  w^R(t,x)=
  \left\{
  \begin{aligned}
    &w_- && \quad{\rm for}\ x\leq w_-t, \\
    &{x}/{t} && \quad{\rm for}\ w_-t<x<w_+t,\\
    &w_+&& \quad{\rm for}\ x\geq w_+t.
  \end{aligned}
  \right.
\end{equation*}
The main idea in \cite{HMS03SIMAMR2000974} is to approximate $w^R(t, x)$
by the solution $w(t, x)$ of the Cauchy problem \eqref{Burgers}.
The following
lemma can be deduced
by virtue of the method of characteristics
(see \cite{MN04book,HMS03SIMAMR2000974}).
\begin{lemma}
  \label{lem_Burgers}
  Let $w_-<w_+$. Then the Burgers equation \eqref{Burgers} has
  a unique smooth solution $w(t,x)$ satisfying
  \begin{list}{}{\setlength{\parsep}{\parskip}
             \setlength{\itemsep}{0.1em}
             \setlength{\labelwidth}{2em}
             \setlength{\labelsep}{0.4em}
             \setlength{\leftmargin}{2.2em}
             \setlength{\topsep}{1mm}
             }
    \item[{\rm(i)}]  $w_x(t,x)\geq 0,$ $w_-\leq w(t,x)<w_+$
    for $(t,x)\in\mathbb{R}_+\times\mathbb{R}$;
    \item[{\rm(ii)}] when $x\leq w_-t$, $w(t,x)-w_-=w_x(t,x)=w_{xx}(t,x)=0$;
    \item[{\rm(iii)}] for each $p\in[1,\infty]$, there exists a constant $C_{p,q}$ such
    that
    \begin{align*}
      \left\|w_{x}(t)\right\|_{L^{p}(\mathbb{R})}
      &\leq C_{p,q}\min\{\tilde{w},\tilde{w}^{\frac{1}{p}}t^{-1+\frac{1}{p}}\},\\
      \left\|w_{xx}(t)\right\|_{L^{p}(\mathbb{R})}
      &\leq C_{p,q}\min\{\tilde{w},t^{-1+\frac{1}{q}(1-\frac{1}{p})} \};
    \end{align*}
    \item[{\rm(iv)}] $\lim_{t\to\infty}\sup_{x\in\mathbb{R}}
    |w(t,x)-w^R(t,x)|=0$.
  \end{list}
\end{lemma}
Having obtained $w(t, x)$,
we can define the smoothed rarefaction wave  $(\bar{\rho},\bar{u},\bar{\theta})$
according to \eqref{SRW}. Then one can check from a direct calculation that
$(\bar{\rho},\bar{u},\bar{\theta})$ solves
the compressible Euler system \eqref{Euler} and
\begin{equation}\label{SRW_p1}
  (\bar{\rho},\bar{u},\bar{\theta})(0,t)=
  ({\rho}_-,{u}_-,{\theta}_-),\quad
  \lim_{x\to\infty}(\bar{\rho},\bar{u},\bar{\theta})(t,x)
  =({\rho}_+,{u}_+,{\theta}_+)\quad {\rm for\ each}\ t\geq 0.
\end{equation}
In view of \eqref{SRW} and
Lemma \ref{lem_Burgers}, we have the following properties for
the smoothed rarefaction wave $(\bar{\rho},\bar{u},\bar{\theta})$.
\begin{lemma}
  \label{lem_SRW}
  The smooth approximation  $(\bar{\rho},\bar{u},\bar{\theta})$
connecting $(\rho_{-},u_{-},\theta_{-})$ and
 $(\rho_{+},u_{+},\theta_{+})$ satisfies
  \begin{list}{}{\setlength{\parsep}{\parskip}
             \setlength{\itemsep}{0.1em}
             \setlength{\labelwidth}{2em}
             \setlength{\labelsep}{0.4em}
             \setlength{\leftmargin}{2.2em}
             \setlength{\topsep}{1mm}
             }
    \item[{\rm(i)}] $\rho_-\leq \bar\rho\leq \rho_+,$ $\theta_-\leq \bar\theta
    \leq \theta_+,$
    $\bar{\theta}_x(t,x)\geq 0,$ and
       \begin{equation}\label{SRW_p2}
         \bar{\rho}_x=\frac{1}{\gamma-1}\bar{\rho}\bar{\theta}^{-1}\bar{\theta}_x,\quad
                     \bar{u}_x= \frac{\sqrt{R\gamma  }}{\gamma-1}\bar{\theta}^{-\frac{1}{2}}
                     \bar{\theta}_x;
       \end{equation}
    \item[{\rm(ii)}] if $x\leq (u_-+\sqrt{R\gamma\theta_-}
    )(1+t)$, then $(\bar{\rho},\bar{u},\bar{\theta})(t,x)=(\rho_{-},u_{-},\theta_{-})$;
    \item[{\rm(iii)}] for each $p\in[1,\infty]$, there exists a constant $C_{p,q}$ such
    that
    \begin{align}\label{SRW_p3}
      \left\|(\bar{\rho}_{x},\bar{u}_{x},\bar{\theta}_{x})(t)\right\|
      _{L^{p}(\mathbb{R}_+)}
      &\leq C_{p,q}\min \{\delta,\delta^{\frac{1}{p}}(1+t)^{-1+\frac{1}{p}} \},
      \\ \label{SRW_p4}
      \left\|(\bar{\rho}_{xx},\bar{u}_{xx},\bar{\theta}_{xx})(t)\right\|
      _{L^{p}(\mathbb{R}_+)}
      &\leq C_{p,q}\min \{\delta,(1+t)^{-1+\frac{1}{q}(1-\frac{1}{p})}\},
    \end{align}
    where
    $\delta:=\left|(u_+-u_-,\theta_+-\theta_-)\right|$ is
    the boundary strength;
    \item[{\rm(iv)}] $\lim_{t\to\infty}\sup_{x\in\mathbb{R}_+}
    |(\bar{\rho},\bar{u},\bar{\theta})(t,x)-({\rho}^R,{u}^R,{\theta}^R)(1+t,x)|=0$.
  \end{list}
\end{lemma}

Next
we state the existence and the properties of the stationary solution
$(\tilde{\rho},\tilde{u},\tilde{\theta})$
satisfying \eqref{BL1} and \eqref{BL2}, which has been derived in
\cite{KNNZMR2755498}. To this end,
we introduce the Mach number at infinity as
$$
M_m:=\frac{|u_m|}{c_m},
$$
where $c_m:=\sqrt{R\gamma\theta_m}$ is the sound speed.
\begin{lemma}[\cite{K82MR651877}]
  Suppose that $(u_-,\theta_-)$ satisfies
  \begin{equation*}
   (u_-,\theta_-)\in \mathcal{M}^m:=\left\{(u,\theta)\in\mathbb{R}^2:
    \left|(u-u_m,\theta-\theta_m)\right|<\delta_0\right\}
  \end{equation*}
  for a certain positive constant $\delta_0$.
\begin{list}{}{\setlength{\parsep}{\parskip}
             \setlength{\itemsep}{\parskip}
             \setlength{\labelwidth}{2em}
             \setlength{\labelsep}{0.4em}
             \setlength{\leftmargin}{2.2em}
             \setlength{\topsep}{1mm}
             }
    \item[{\rm(i)}] For the case $M_m>1$, there exists a unique smooth solution
  $(\tilde{\rho},\tilde{u},\tilde{\theta})$ to the problem \eqref{BL1}-\eqref{BL2}
  satisfying the decay estimate \eqref{decay1}.
  \item[{\rm(ii)}] For the case $M_m=1$,
   there exists a certain region $\mathcal{M}^0\subset \mathcal{M}^m$
   such that if $(u_-,\theta_-)\in \mathcal{M}^0$, then there exists
   a unique smooth solution $(\tilde{\rho},\tilde{u},\tilde{\theta})$
to \eqref{BL1}-\eqref{BL2}
satisfying
  \begin{equation*} %\label{decay2}
   |\partial_x^n(\tilde{\rho}-\rho_m,\tilde{u}-u_m,\tilde{\theta}-\theta_m)(x)|
   \leq \frac{C\tilde{\delta}^{n+1}}{(1+\tilde{\delta} x)^{k+1}}
   +C\tilde{\delta} \mathrm{e}^{-cx} \quad {\rm for\ all}\ n\in\mathbb{N}.
  \end{equation*}
  \item[{\rm(iii)}] For the case $M_m<1$, there exists a curve $\mathcal{M}^-\subset
  \mathcal{M}^m$
  such that if
  $
   (u_-,\theta_-)\in \mathcal{M}^m,
  $
  then there exists a unique smooth solution $(\tilde{\rho},\tilde{u},\tilde{\theta})$
to the problem \eqref{BL1}-\eqref{BL2}
  satisfying \eqref{decay1}.
\end{list}
\end{lemma}
Since the stationary solution $(\tilde{\rho},\tilde{u},\tilde{\theta})$
and the smoothed rarefaction wave $(\bar{\rho},\bar{u},\bar{\theta})$
are well-defined,
one can deduce that $(\hat\rho,\hat{u},\hat\theta)$ satisfies
\begin{equation}\label{hat1}\left\{
  \begin{aligned}
    \hat{\rho}_t+\hat{u}\hat{\rho}_x+\hat{\rho}\hat{u}_x&=\hat{f_1},\\
    \hat{\rho}(\hat{u}_t+\hat{u} \hat{u}_x)+\hat{P}_x&=\mu\tilde{u}_{xx}+
    \hat{f_2},\\
    c_v\hat{\rho}(\hat{\theta}_t+\hat{u} \hat{\theta}_x)+\hat{P}\hat{u}_x&
    =\kappa\tilde{\theta}_{xx}+\mu\tilde{u}_x^2+\hat{f_3}
  \end{aligned}\right.
\end{equation}
for $(t,x)\in\mathbb{R}_+\times\mathbb{R}$ and the condition
\begin{equation}\label{hat2}
  (\hat\rho,\hat{u},\hat\theta)(t,0)=( {\rho}_-, {u}_-, {\theta}_-),\quad
  \lim_{x\to\infty}(\hat\rho,\hat{u},\hat\theta)(t,x)=( {\rho}_+, {u}_+, {\theta}_+)
  \quad {\rm for\ each}\ t\geq 0.
\end{equation}
Here $\hat{P}:=P(\hat{\rho},\hat{\theta})=R\hat{\rho}\hat{\theta}$ and
\begin{align}
\label{f1_hat}
  \hat{f_1}=~&\tilde{u}_{x}(\bar{\rho}-\rho_m)+\bar{u}_{x}(\tilde{\rho}-\rho_m)
  +\tilde{\rho}_{x}(\bar{u}-u_m)+\bar{\rho}_{x}(\tilde{u}-u_m),\\
  \hat{f_2}=~&\hat{\rho}[\tilde{u}_{x}(\bar{u}-u_m)+\bar{u}_{x}(\tilde{u}-u_m)]
  +\tilde{u}\tilde{u}_{x}(\bar{\rho}-\rho_m)\nonumber\\  \label{f2_hat}
  &+ (\hat{P}-\tilde{P}-\bar{P})_x
  -\bar{\rho}^{-1}(\tilde{\rho}-\rho_m)\bar{P}_x,\\
  \hat{f_3}=~&c_v\hat{\rho}[\tilde{\theta}_{x}(\bar{u}-u_m)+\bar{\theta}_{x}(\tilde{u}-u_m)]
  +c_v(\bar{\rho}-\rho_m)\tilde{u}\tilde{\theta}_{x}\nonumber\\ \label{f3_hat}
  &+\tilde{u}_x(\hat{P}-\tilde{P} )+(\hat{P} -\bar{P})\bar{u}_x
  -R\bar{\theta}(\tilde{\rho}-\rho_m)\bar{u}_x.
\end{align}

\section{Stability analysis}
This section is devoted to proving our main results: Theorem  \ref{thm1} and
Theorem \ref{thm2}. We will concentrate on the proof of Theorem \ref{thm2},
that is, the stability of the composition of a rarefaction wave and
a non-degenerate stationary solution. The proof of Theorem \ref{thm1}
is similar to and simpler than that of Theorem \ref{thm2}.
We therefore omit it here for brevity.

First we introduce the perturbation $(\phi,\psi,\vartheta)$
toward the superposition wave $(\hat{\rho},\hat{u},\hat{\theta})$ as
\begin{equation*}%\label{per}
  (\phi, \psi, \vartheta)(t,x):=(\rho, u, \theta)(t,x)-(\hat{\rho},
  \hat{u}, \hat{\theta})(t,x),
\end{equation*}
where  $(\hat{\rho},\hat{u},\hat{\theta})$ is given by \eqref{hat}.
Then we subtract \eqref{hat1}-\eqref{hat2} from \eqref{NS_E}-\eqref{bdy}
to have the initial boundary value problem:
\begin{equation}\label{per1}
\left\{
  \begin{aligned}
    \phi_t+u\phi_x+\rho\psi_x&=f_1,\\[0.5mm]
    \rho(\psi_t+u\psi_x)+(P-\hat{P})_x&=\mu\psi_{xx}+f_2,\\[0.5mm]
    c_v\rho(\vartheta_t+u\vartheta_x)+P\psi_x
    &=\kappa\vartheta_{xx}+\mu\psi_x^2+f_3
  \end{aligned}\right.
\end{equation}
for $(t,x)\in\mathbb{R}_{+}\times\mathbb{R}_{+}$ with the initial
and boundary conditions
\begin{equation}\label{per2}
    (\phi, \psi, \vartheta)|_{t=0}=(\phi_0, \psi_0, \vartheta_0),\quad
    (\psi, \vartheta)|_{x=0}=(0,0).
\end{equation}
Here the initial condition
$(\phi_0, \psi_0, \vartheta_0)
    :=(\rho_0,u_0,\theta_0)-(\hat{\rho},\hat{u},\hat{\theta})|_{t=0}$
    satisfies
\begin{equation}\label{per3}
  \lim_{x\to\infty}(\phi_0, \psi_0, \vartheta_0)(x)
    =(0,0,0),
\end{equation}
and
\begin{align}\label{f1}
  f_1 =~&
  -\hat{u}_x\phi-\hat{\rho}_x\psi-\hat{f_1},  \\
  \label{f2}
  f_2 =~  &
  \mu\bar{u}_{xx}-\mu\hat{\rho}^{-1}\phi\tilde{u}_{xx}
  +\hat{\rho}^{-1}\phi  \hat{P}_x-\rho\psi\hat{u}_x-\hat{\rho}^{-1}\rho \hat{f_2},\\
  f_3 =~  &
  \kappa\bar{\theta}_{xx}-\hat{\rho}^{-1}\phi (\kappa\tilde{\theta}_{xx}+
  \mu\tilde{u}_{x}^2)+\mu\bar{u}_x^2\nonumber\\ \label{f3}
  &+2\mu\psi_x\hat{u}_x+2\mu\tilde{u}_x\bar{u}_x
  -R\rho\vartheta\hat{u}_x-c_v\rho\hat{\theta}_x\psi-\hat{\rho}^{-1}\rho \hat{f_3},
\end{align}
where $\hat{f_i}$ $(i=1,2,3)$ are defined by \eqref{f1_hat}-\eqref{f3_hat},
respectively.

We turn to deduce some desired a priori estimates for the
perturbation $(\phi,\psi,\vartheta)$ in the Sobolev space $H^1$.
Before doing so,
for some non-negative constants $N$, $s$, $t$ and  $m_i$ $(i=1,2)$ with $t\geq s$,
we introduce
the set in which
we seek the solution of the initial boundary value problem
\eqref{per1}-\eqref{per2} as follows
\begin{equation*}
  \begin{split}
     X(s,t ;m_1,m_2,N):=\big\{(\phi,\psi,\vartheta)
      \in C([s,t ];H^1):
     (\psi_x,\vartheta_x)\in L^2(s,t ;H^1),
     \phi_x\in L^2(s,t ;L^2),\\
          \|(\phi,\psi,\vartheta)(t)\|_1\leq N,\
     (\phi+\hat{\rho})(t,x)\geq m_1,\
     (\vartheta+\hat{\theta})(t,x)\geq m_2\ \forall\, (t,x)\in[s,t ]\times\mathbb{R}_+
     \big\}.
  \end{split}
\end{equation*}
The letter $C$ or $C_i$ $(i\in\mathbb{N})$ will be employed to denote
some positive constant which depends only on
$\inf_{x\in\mathbb{R}_+}\left\{\rho_0(x),\theta_0(x)\right\}$
  and $\|(\phi_0,\psi_0,\vartheta_0)\|_1$.
The exact value denoted by $C$ or $C_i$ may therefore vary from line to line.
  For notational simplicity, we introduce
  $A\lesssim B$ if $A\leq C B$ holds uniformly for some constant  $C$.
  The notation $A\sim B$ means that both $A\lesssim B$  and $B\lesssim A$.
  Besides, we will use the notation $(\rho, \theta)=(\phi+\hat{\rho},
  \vartheta+\hat{\theta})$.

 To make the presentation clearly, we divide this section into five parts.
 The first four parts concern  the a priori estimates for
 the solution $(\phi,\psi,\vartheta)\in X(0,T;m_1,m_2,N)$ to
 the problem \eqref{per1}-\eqref{per2}, where
 $T>0$ and it will be assumed that
 $m_i\leq 1\leq N$ ($i=1,2$) so that
  \begin{equation}\label{apriori}
    \|(\phi,\psi,\vartheta)(t)\|_1\leq N,\ \
    m_1\leq {\rho}(t,x)\lesssim N,\ \
  m_2\leq   {\theta}(t,x)\lesssim N \ \ {\rm for\ all}\
  (t,x)\in[0,T]\times\mathbb{R}_+.
  \end{equation}
In Subsection \ref{sec_proof},
the last part of this section, we will combine
the energy estimates with a well-designed
continuation argument to prove Theorem \ref{thm2}.

\subsection{Basic energy estimate} \label{sec_bas}
In this part, we will show the following basic energy estimate.
\begin{lemma}\label{lem_bas}
Suppose that the conditions listed in Theorem \ref{thm2} hold.
Then there exists a sufficiently small $\epsilon_0>0$ such that if
 \begin{equation}\label{asmp1}
   \Xi(m_1, m_2,N) (
   \bar{\delta}+\tilde{\delta})
   \leq \epsilon_0
 \end{equation}
  with $\Xi(m_1, m_2 ,N):=  m_1^{-50}m_2^{-50}N^{50},$
 then
 \begin{align}
    \label{est1}
    \sup_{0\leq t\leq T}\int_{\mathbb{R}_+}\rho \mathcal{E}\mathrm{d}x
    +\int_0^T\rho\Phi\left(\frac{\hat{\rho}}{\rho}\right)(t,0)\mathrm{d}t
    +\int_0^T\int_{\mathbb{R}_+}\left[\frac{\psi_x^2}{\theta}+
    \frac{\vartheta_x^2}{\theta^2}\right]\mathrm{d}x\mathrm{d}t
    &\lesssim 1,\\
    \label{est2}
    \sup_{0\leq t\leq T}\int_{\mathbb{R}_+}\frac{\phi_x^2}{\rho^3}\mathrm{d}x
    +\int_0^T\frac{\phi_x^2}{\rho^3}(t,0)
    \mathrm{d}t
    +\int_0^T\int_{\mathbb{R}_+}\frac{\theta\phi_x^2}{\rho^2}
    \mathrm{d}x\mathrm{d}t
    &\lesssim N,
  \end{align}
  where
  \begin{equation}\label{E}
      \mathcal{E}:=R\hat{\theta}\Phi\left(\frac{\hat{\rho}}{\rho}\right)
      +\frac12\psi^2 +c_v\hat{\theta}\Phi\left(\frac{\theta}{\hat{\theta}}\right),
  \quad \Phi(z):= z-\ln z-1.
  \end{equation}
\end{lemma}
\begin{proof}

\noindent{\bf Step 1}.
After a straightforward calculation, one can derive
  \begin{equation}\label{id1}
  (\rho \mathcal{E})_t+\left[\rho u\mathcal{E}+\psi(P-\hat{P})-\mu\psi\psi_x
  -\kappa\frac{\vartheta\vartheta_x}{\theta}\right]_x
  +\mu\frac{\hat{\theta}\psi_x^2}{\theta}
  +\kappa\frac{\hat{\theta}\vartheta_x^2}{\theta^2}=\sum_{q=1}^{5}\mathcal{R}_q,
  \end{equation}
  from which we obtain
  \begin{equation}\label{id1a}
  \frac{\mathrm{d}}{\mathrm{d}t}\int_{\mathbb{R}_+}\rho \mathcal{E}\mathrm{d}x
   - u_- \rho \Phi\left(\frac{\hat{\rho}}{\rho} \right)(t,0)
  +\int_{\mathbb{R}_+}\left[\mu\frac{\hat{\theta}\psi_x^2}{\theta}
  +\kappa\frac{\hat{\theta}\vartheta_x^2}{\theta^2}\right]\mathrm{d}x
  =\sum_{q=1}^{5}\int_{\mathbb{R}_+}\mathcal{R}_q\mathrm{d}x,
  \end{equation}
  where each term $\mathcal{R}_q$
  on the right-hand side of \eqref{id1} will be defined below.
Before defining and estimating all the terms
on the right-hand side of \eqref{id1a},
we set
\begin{align*}
   \bar{U}:=(\bar{\rho},\bar{u},\bar{\theta}),\quad
   \tilde{U}:=(\tilde{\rho},\tilde{u},\tilde{\theta}),\quad
   U_m:=({\rho}_m,{u}_m,{\theta}_m),\quad
   \Psi :=(\phi,\psi,\vartheta).
\end{align*}
First we consider
  \begin{equation*}            %\label{R1}
     \mathcal{R}_1:=
     \frac{\vartheta}{\theta}
     \left[\kappa\frac{\bar{\theta}_x\vartheta_x}{\theta}
     +2\mu\bar{u}_x\psi_x\right],
  \end{equation*}
which is trivially estimated by Sobolev's inequality as
\begin{equation*}         %\label{R1a}
   |\mathcal{R}_1 |
  \lesssim m_2^{-2}  |\bar{U}_x | |\Psi | |\Psi_x|
  \lesssim m_2^{-2}  \|\Psi \|^{\frac{1}{2}} \|\Psi_x \|^{\frac{1}{2}}
   |\bar{U}_x | |\Psi_x |.
\end{equation*}
In light of \eqref{SRW_p3} and \eqref{apriori},
we apply H\"{o}lder's and Young's inequalities
to deduce
\begin{equation}\label{R1i}
\begin{aligned}
  \int_{\mathbb{R}_+} |\mathcal{R}_1 |\mathrm{d}x\lesssim~&
   m_2^{-2}   \|\Psi \|^{\frac{1}{2}} \|\bar{U}_x \|
   \|\Psi_x \|^{\frac{3}{2}}\\
  \lesssim~& m_2^{-2}N^{\frac{1}{2}}  \bar{\delta}^{\frac{1}{2}}
  (1+t)^{-\frac{1}{2}}  \|\Psi_x \|^{\frac{3}{2}}\\[1mm]
  \lesssim~&  (1+t)^{-2}+m_2^{-\frac{8}{3}}N^{\frac{2}{3}}
  \bar{\delta}^{\frac{2}{3}} \|\Psi_x \|^2.
\end{aligned}
\end{equation}
Next we consider the term
  \begin{equation*}             %\label{R2}
     \mathcal{R}_2:=
    \frac{\vartheta}{\theta}(\kappa\bar{\theta}_{xx}+\mu\bar{u}_x^2)
    +\mu\psi\bar{u}_{xx}.
  \end{equation*}
This term can be controlled by Sobolev's inequality as
\begin{equation*}%\label{R2a}
   |\mathcal{R}_2 |
  \lesssim m_2^{-1}  |\Psi |\left[ |\bar{U}_{xx} |+
   |\bar{U}_{x} |^2\right]
  \lesssim m_2^{-1}  \|\Psi \|^{\frac{1}{2}}\|\Psi_x\|^{\frac{1}{2}}
  \left[|\bar{U}_{xx} |+|\bar{U}_{x} |^2\right].
\end{equation*}
According to \eqref{SRW_p3}, \eqref{SRW_p4} and
\eqref{apriori}, we infer
\begin{equation}\label{R2i}
\begin{aligned}
  \int_{\mathbb{R}_+} |\mathcal{R}_2 |\mathrm{d}x\lesssim~&
   m_2^{-1} \|\Psi\|^{\frac{1}{2}}  \|\Psi_{x}\|^{\frac{1}{2}}
  \left[\|\bar{U}_{xx}\|_{L^1}+
  \|\bar{U}_{x}\|^2\right]\\
  \lesssim~& m_2^{-1}N^{\frac{1}{2}}\|\Psi_x\|^{\frac{1}{2}}
  \bar{\delta}^{\frac{1}{8}}
  (1+t)^{-\frac{7}{8}}\\[1.5mm]
  \lesssim~&  (1+t)^{-\frac{7}{6}}+m_2^{-4}N^{2}
  \bar{\delta}^{\frac{1}{2}}\|\Psi_x\|^2.
\end{aligned}
\end{equation}
Let us now consider the term
  \begin{equation*}             %\label{R3}
    \mathcal{R}_3:=
    -R\frac{\hat{\theta}\phi\hat{f_1}}{\hat{\rho}}
    +2\mu\frac{\vartheta\bar{u}_x\tilde{u}_x}{\theta}
    -\rho\frac{\vartheta\hat{f_3}+\hat{\theta}\psi\hat{f_2}}{\hat{\rho}\hat{\theta}},
  \end{equation*}
It is not difficult to derive from \eqref{f1_hat}-\eqref{f3_hat} that
\begin{equation} \label{f_hat1}
   | (\hat{f_1},\hat{f_2},\hat{f_3} ) |
  \lesssim  |\bar{U}_x | |\tilde{U}-U_m |
  + |\tilde{U}_x | |\bar{U}-U_m |.
\end{equation}
In view of  Lemma \ref{lem_SRW}, we deduce that
$\bar{U}(t,0)\equiv U_m$ and hence we have  that for $q\geq 1$,
\begin{equation*}
  \begin{aligned}
      &\left\||\bar{U}_x||\tilde{U}-U_m|
  +|\tilde{U}_x||\bar{U}-U_m|+
  |\bar{U}_x||\tilde{U}_x|\right\|_{L^q}\\
  &\quad \lesssim
  \left\||\bar{U}_x||\tilde{U}-U_m|
  +|\tilde{U}_x|
  \int_0^x|\bar{U}_x|(\cdot,y)\mathrm{d}y+
  |\bar{U}_x||\tilde{U}_x|\right\|_{L^q}\\
  &\quad \lesssim
  \|\bar{U}_x\|_{L^{\infty}}
 \left\||\tilde{U}-U_m|
  +x|\tilde{U}_x|+
  |\tilde{U}_x|
  \right\|_{L^q},
  \end{aligned}
\end{equation*}
which combined with \eqref{decay1} implies
\begin{equation}\label{point1}
     \left\||\bar{U}_x||\tilde{U}-U_m|
  +|\tilde{U}_x||\bar{U}-U_m|+
  |\bar{U}_x||\tilde{U}_x|\right\|_{L^q}
  \lesssim  \tilde{\delta}  \|\bar{U}_x\|_{L^{\infty}}.
\end{equation}
It follows from
\eqref{apriori}, \eqref{f_hat1}
and \eqref{point1} with $q=1$ that
\begin{equation*}
  \int_{\mathbb{R}_+}|\mathcal{R}_3|\mathrm{d}x
  \lesssim N m_2^{-1}
  \|\Psi\|_{L^{\infty}}   \tilde{\delta}\|\bar{U}_x\|_{L^{\infty}}.
\end{equation*}
We use \eqref{SRW_p3} and Young's inequality
to have
\begin{equation} \label{R3i}
\begin{aligned}
  \int_{\mathbb{R}_+}|\mathcal{R}_3|\mathrm{d}x
  \lesssim N m_2^{-1}\tilde{\delta}(1+t)^{-1}
  \|\Psi\|^{\frac{1}{2}}  \|\Psi_{x}\|^{\frac{1}{2}}
  \lesssim
  (1+t)^{-\frac{4}{3}}+ m_2^{-4}\tilde{\delta}^{4}N^6\|\Psi_{x}\|^2.
\end{aligned}
\end{equation}
We then estimate the term
\begin{equation*}  %\label{R4}
  \begin{aligned}
        \mathcal{R}_4:=~&
    \psi\tilde{\theta}_x\left[-\frac{c_v}{\hat{\theta}}\rho\vartheta
    +R\phi+R\rho\Phi\left(\frac{\hat{\rho}}{\rho}\right)
    +c_v\rho\Phi\left(\frac{\theta}{\hat{\theta}}\right)\right] \\
    &-\rho\tilde{u}_x\left[\frac{R^2\hat{\theta}}{c_v}
    \Phi\left(\frac{\hat{\rho}}{\rho}\right)+\psi^2
    +R\hat{\theta}\Phi\left(\frac{\theta}{\hat{\theta}}\right)\right]
    +\mu\psi\tilde{u}_{xx}\left[1-\frac{\rho}{\hat{\rho}}\right]
    \\
    &+\vartheta\left[\frac{1}{\theta}-\frac{\rho}{\hat{\rho}\hat{\theta}}\right]
    \left[\kappa\tilde{\theta}_{xx}+\mu\tilde{u}_x^2\right]
    +\frac{\vartheta}{\theta}
    \left[\kappa\frac{\tilde{\theta}_x\vartheta_x}{\theta}+2\mu\tilde{u}_x\psi_x\right]
    \\
    &+\frac{\rho}{\hat{\rho}}
    \left[\kappa\tilde{\theta}_{xx}+\mu\tilde{u}_x^2+\hat{f_3}\right]
    \left[\frac{R}{c_v} \Phi\left(\frac{\hat{\rho}}{\rho}\right)
    + \Phi\left(\frac{\theta}{\hat{\theta}}\right)\right].
  \end{aligned}
\end{equation*}
To this end, we first obtain from the identity
\begin{equation*}
    \Phi(z)=\int_0^1\int_0^1\theta_1\Phi''(1+\theta_1\theta_2(z-1))
    \mathrm{d}\theta_2\mathrm{d}\theta_1
    (z-1)^2
\end{equation*}
that
\begin{equation}\label{Phi1}
  (z+1)^{-2}(z-1)^2\lesssim\Phi(z)\lesssim (z^{-1}+1)^2(z-1)^2.
\end{equation}
This last inequality implies
\begin{equation} \label{Phi2}
  \Phi\left(\frac{\hat{\rho}}{\rho}\right) \lesssim m_1^{-2}\phi^2
  \lesssim m_1^{-2}N|\phi|,\quad
    \Phi\left(\frac{\theta}{\hat{\theta}}\right) \lesssim m_2^{-2}\vartheta^2
  \lesssim m_2^{-2}N|\vartheta|.
\end{equation}
In light of \eqref{f_hat1} and \eqref{Phi2},
we discover
\begin{equation} \label{R4a}
    \int_{\mathbb{R}_+}|\mathcal{R}_4|\mathrm{d}x
  \lesssim N^2 m_1^{-2}m_2^{-2}\sum_{\ell=0}^2\sum_{k=0}^1
  \int_{\mathbb{R}_+}\big|\partial_x^{\ell}(\tilde{U}-U_m)\big|
  |\partial_x^{k}\Psi|^2\mathrm{d}x.
\end{equation}
To estimate the terms on the right-hand side of \eqref{R4a},
we utilize an idea in Nikkuni-Kawashima \cite{NK00MR1793167}, that is,
the following Poincar\'{e} type inequality
$$
|\varphi(t,x)|\leq|\varphi(t,0)|+\sqrt{x}\|\varphi_x(t)\|
\quad {\rm for}\ x\in\mathbb{R}_+.
$$
Applying this inequality to $\Psi$, we deduce
from  \eqref{decay1} and \eqref{Phi1} that
\begin{equation}\label{R4b}
\begin{split}
  \int_{\mathbb{R}_+}|\partial_x^{\ell}(\tilde{U}-U_m)|
  |\partial_x^{k}\Psi|^2\mathrm{d}x
  \lesssim \tilde{\delta} \phi(t,0)^2
  +\tilde{\delta}\|\Psi_{x}\|^2
   \lesssim
  N^2 m_1^{-1}\tilde{\delta}
  \rho\Phi\left(\frac{\hat{\rho}}{\rho}\right) (t,0)
  +\tilde{\delta}\|\Psi_{x}\|^2
\end{split}
\end{equation}
 for $k=0,1$ and  $\ell\in\mathbb{N}$.
Plug \eqref{R4b} into \eqref{R4a} to deduce
 \begin{equation}\label{R4i}
\begin{split}
   \int_{\mathbb{R}_+} |\mathcal{R}_4 |\mathrm{d}x
  \lesssim
  N^4  m_1^{-3}m_2^{-2}\tilde{\delta}
  \rho\Phi\left(\frac{\hat{\rho}}{\rho}\right) (t,0)
  +N^2 m_1^{-2}m_2^{-2}\tilde{\delta} \|\Psi_{x}\|^2.
\end{split}
\end{equation}
We next estimate the term
  \begin{equation*}     %\label{R5}
  \begin{aligned}
        \mathcal{R}_5:=~&
    \psi\bar{\theta}_x\left[-\frac{c_v}{\hat{\theta}}\rho\vartheta
    +R\phi+R\rho\Phi\left(\frac{\hat{\rho}}{\rho}\right)
    +c_v\rho\Phi\left(\frac{\theta}{\hat{\theta}}\right)\right] \\
    &-\rho\bar{u}_x\left[\frac{R^2\hat{\theta}}{c_v}
    \Phi\left(\frac{\hat{\rho}}{\rho}\right)+\psi^2
    +R\hat{\theta}\Phi\left(\frac{\theta}{\hat{\theta}}\right)\right].
  \end{aligned}
  \end{equation*}
To this end, we introduce
\begin{equation}\label{ab}
  a:=\ln\left(\frac{\hat{\rho}}{\rho}\right)\quad {\rm and}\quad
  b:=\ln\left(\frac{\theta}{\hat{\theta}}\right).
\end{equation}
In light of  \eqref{SRW_p2}, one can find
\begin{equation}\label{R5a}
  \mathcal{R}_5=-\rho\bar{\theta}_x F(\psi,a,b)
\end{equation}
with
\begin{equation*} %\label{F}
  \begin{aligned}
    F(\psi,a,b):=~&
    R\psi a+c_v\psi b+
    \sqrt{{R}{\gamma}}R\bar{\theta}^{-\frac{1}{2}}\hat{\theta}
     (\mathrm{e}^a-a-1 )
    \\
    &+\frac{\sqrt{R\gamma  }}{\gamma-1}\bar{\theta}^{-\frac{1}{2}}\psi^2
    +\sqrt{R\gamma}c_v\bar{\theta}^{-\frac{1}{2}}\hat{\theta}
     (\mathrm{e}^b-b-1 ).
  \end{aligned}
\end{equation*}
One can easily deduce that
$
  (F,\partial_{\psi}F,\partial_{a}F,\partial_{b}F)(0,0,0)=(0,0,0,0)
$
and that the Hessian matrix of $F$ at point $(0,0,0)$ is
\begin{equation*}
  H_{F} =
  \begin{pmatrix}
    2\frac{\sqrt{R\gamma  }}{\gamma-1}\bar{\theta}^{-\frac{1}{2}} &
    R&c_v\\[1mm]
    R&\sqrt{  R\gamma } R\bar{\theta}^{-\frac{1}{2}}\hat{\theta}&0\\[1mm]
    c_v&0&\sqrt{ R\gamma}c_v\bar{\theta}^{-\frac{1}{2}}\hat{\theta}
  \end{pmatrix}.
\end{equation*}
Let $D_k$ $(k=1,2,3)$ be
the $k\times k$ leading principal minor of $H_{F} $. Then we have
\begin{equation*}
  D_1=2\frac{\sqrt{R\gamma  }}{\gamma-1}\bar{\theta}^{-\frac{1}{2}},\ \
  D_2=R^2\frac{2\gamma\bar{\theta}^{-1}\hat{\theta}-\gamma+1}{\gamma-1},\ \
  D_3=c_v\sqrt{\gamma R}R^2\bar{\theta}^{-\frac{1}{2}}\hat{\theta}
  \frac{2\gamma\bar{\theta}^{-1}\hat{\theta}-\gamma
  }{\gamma-1}.
\end{equation*}
Apply Sylvester's criterion (see \cite[Theorem 7.2.5]{HJ13MR2978290}) to
deduce that there exists a positive constant $\varepsilon_1$ such that
if $\tilde{\delta}\leq \varepsilon_1$, then
$H_{F} $ is positive definite.
Then it follows from the Taylor formula
(see \cite[Theorem VII.5.8]{A08MR2419362}) that
\begin{equation*} %\label{F1}
   \begin{aligned}
     F(\psi,a,b) =  ~&\sum_{q=0}^2
     \frac{1}{q!}(\psi\partial_{\psi}+a\partial_{a}+b\partial_{b})^q
     F(0,0,0)\\
     & +\frac{1}{2}\int_0^1(1-t)^2(\psi\partial_{\psi}+a\partial_{a}+b\partial_{b})^3
     F(t\psi,ta,tb)\mathrm{d}t
%     \\ =~&\frac{1}{2}(\psi,a,b)H_{F} (\psi,a,b)^{T}
%     +\frac{1}{2}\sqrt{\gamma R}\bar{\theta}^{-\frac{1}{2}}\hat{\theta}
%     \int_0^1(1-t)^2
%     \left(Ra^3\mathrm{e}^{ta}+c_vb^3\mathrm{e}^{tb}\right)\mathrm{d}t
     \\ \leq ~&
     \frac{1}{2}(\psi,a,b)H_{F} (\psi,a,b)^{T}
     +\frac{1}{2}\sqrt{R\gamma }\bar{\theta}^{-\frac{1}{2}}\hat{\theta}
     \int_0^1\left(Ra^3\mathrm{e}^{ta}+c_vb^3\mathrm{e}^{tb}\right)\mathrm{d}t
         .
   \end{aligned}
\end{equation*}
Due to $\bar{\theta}_x\geq 0$, we derive that
if $\tilde{\delta}\leq \varepsilon_1 $,
then
\begin{equation}\label{R5b}
  \mathcal{R}_5\lesssim \rho\bar{\theta}_x
  \left[a^2(\mathrm{e}^{|a|}-1)+b^2(\mathrm{e}^{|b|}-1)\right].
\end{equation}
We note that the identity
\begin{equation*}
    \ln z=\int_0^1\frac{(z-1)\mathrm{d}\theta_1}{1+\theta_1(z-1)}
\end{equation*}
implies
\begin{equation}\label{ln}
  |\ln z|^2\lesssim (z^{-1}+1)^2(z-1)^2.
\end{equation}
Then we apply \eqref{ln} to $a$ and $b$ and use the estimate
\eqref{R5b} to find
\begin{equation*} %\label{R5c}
  \mathcal{R}_5\lesssim N m_{1}^{-3}m_{2}^{-3}\bar{\theta}_x
  \left|\Psi\right|^3,
\end{equation*}
which combined with \eqref{SRW_p3} yields
\begin{equation}\label{R5i}
  \begin{aligned}
  \int_{\mathbb{R}_+} |\mathcal{R}_5 |\mathrm{d}x
  \lesssim ~&Nm_{1}^{-3}m_{2}^{-3} \|\bar{\theta}_x \|_{L^{\infty}}
  \|\Psi\|_{L^{\infty}}  \|\Psi\|^2 \\
  \lesssim ~&Nm_{1}^{-3}m_{2}^{-3}\bar{\delta}^{\frac{1}{8}}(1+t)^{-\frac{7}{8}}
  \|\Psi\|^{\frac{5}{2}}  \|\Psi_{x}\|^{\frac{1}{2}}\\[1.5mm]
  \lesssim ~&
  (1+t)^{-\frac{7}{6}}+
   m_{1}^{-12}m_2^{-12}\bar{\delta}^{\frac{1}{2}}N^{14}
  \|\Psi_{x}\|^2.
\end{aligned}
\end{equation}
Plug \eqref{R1i}, \eqref{R2i}, \eqref{R3i}, \eqref{R4i}
and \eqref{R5i} into \eqref{id1a} to obtain
  \begin{equation}\label{id1b}
  \begin{aligned}
  &\frac{\mathrm{d}}{\mathrm{d}t}\int_{\mathbb{R}_+}\rho \mathcal{E}\mathrm{d}x
   - u_-\rho \Phi\left(\frac{\hat{\rho}}{\rho}\right)(t,0)
  +\int_{\mathbb{R}_+}\left[ \frac{\psi_x^2}{\theta}
  + \frac{ \vartheta_x^2}{\theta^2}\right]\mathrm{d}x\\
  ~&\quad \lesssim (1+t)^{-\frac{7}{6}}
  + m_{1}^{-12}m_2^{-12}N^{16}
  \left[\bar{\delta}^{\frac{1}{2}}+\tilde{\delta}\right]
  \left[
  \rho\Phi\left(\frac{\hat{\rho}}{\rho}\right) (t,0)+
  \|\Psi_{x}\|^2\right].
  \end{aligned}
  \end{equation}
Hence we can find a sufficiently small  constant
$\varepsilon_2>0$ such that if
\begin{equation}\label{asmp2}
   m_{1}^{-12}m_2^{-12}N^{18}
  \left[\bar{\delta}^{\frac{1}{2}}+\tilde{\delta}\right]\leq \varepsilon_2,
\end{equation}
then
  \begin{equation*}%\label{id1c}
  \begin{aligned}
  &\frac{\mathrm{d}}{\mathrm{d}t}\int_{\mathbb{R}_+}\rho \mathcal{E}\mathrm{d}x
   - \rho u_-\Phi\left(\frac{\hat{\rho}}{\rho}\right)(t,0)
  +\int_{\mathbb{R}_+}\left[ \frac{\psi_x^2}{\theta}
  + \frac{ \vartheta_x^2}{\theta^2}\right]\mathrm{d}x\\[1mm]
  &\quad \lesssim (1+t)^{-\frac{7}{6}}
  + m_{1}^{-12}m_2^{-12}N^{16}
  \left[\bar{\delta}^{\frac{1}{2}}+\tilde{\delta}\right]
  \left\|\phi_{x}\right\|^2,
  \end{aligned}
  \end{equation*}
from which we obtain
  \begin{equation}\label{id1d}
  \begin{aligned}
  &\sup_{0\leq t\leq T}\int_{\mathbb{R}_+}\rho \mathcal{E}\mathrm{d}x
   +\int_0^T
   \rho\Phi\left(\frac{\hat{\rho}}{\rho}\right)(t,0)\mathrm{d}t
  +\int_0^T\int_{\mathbb{R}_+}\left[ \frac{\psi_x^2}{\theta}
  + \frac{ \vartheta_x^2}{\theta^2}\right]   \\
  &\quad \lesssim 1
  + m_{1}^{-12}m_2^{-12}N^{16}
  \left[\bar{\delta}^{\frac{1}{2}}+\tilde{\delta}\right]
  \int_0^T\left\|\phi_{x}(t)\right\|^2\mathrm{d}t.
  \end{aligned}
  \end{equation}

\noindent{\bf Step 2}. We now make some estimates for the last
 term in \eqref{id1d}.
We first differentiate \eqref{per1}$_1$ with respect to $x$
and then multiply the resulting equation by $\phi_x/\rho^3$ to find
\begin{equation}\label{id2}
  \left(\frac{\phi_x^2}{2\rho^3}\right)_t+\left(\frac{u\phi_x^2}{2\rho^3}\right)_x
  -\hat{u}_x\frac{\phi_x^2}{ \rho^3}+\hat{\rho}_x\frac{\phi_x\psi_x}{\rho^3}
  =f_{1x}\frac{\phi_x}{ \rho^3}-\frac{\phi_x\psi_{xx}}{ \rho^2}.
\end{equation}
Multiply \eqref{per1}$_2$ by $\phi_x/\rho^2$ to have
\begin{equation}\label{id3}
  \begin{aligned}
   &\left(\frac{\phi_x\psi}{\rho}\right)_t-
   \left[\frac{\phi_t\psi}{\rho}+\frac{\hat{\rho}_x\psi^2}{\rho}\right]_x
   -\mu  \frac{\phi_x\psi_{xx}}{ \rho^2}   \\
      & \quad =
      -\frac{\rho_x\hat{u}_x\phi\psi}{\rho^2}
      -\frac{\hat{\rho}_{xx}}{\rho}\psi^2
      -\frac{\hat{\rho}_x\psi\psi_x}{\rho}
      -\frac{\hat{\rho}_x\psi\hat{f}_1}{\rho^2}
      -\frac{\rho_x\psi\psi_x}{\rho}\\
     & \quad\quad\,
     +\frac{\hat{u}_x\psi_x\phi}{\rho}
     +\frac{\psi_x}{\rho}\hat{f_1}+\psi_x^2+\frac{u_x\phi_x\psi}{\rho}
     - (P-\hat{P} )_x\frac{\phi_x}{\rho^2}+f_2\frac{\phi_x}{\rho^2}.
  \end{aligned}
\end{equation}
In light of \eqref{id2} and \eqref{id3}, we have
\begin{equation}\label{id4}
      \left[\frac{\mu\phi_x^2}{2\rho^3}+\frac{\phi_x\psi}{\rho}\right]_t
     +\left[\frac{\mu u\phi_x^2}{2\rho^3}
     -\frac{\phi_t\psi}{\rho}-\frac{\hat{\rho}_x\psi^2}{\rho}\right]_x
     +\frac{R\theta\phi_x^2}{\rho^2}
      =\sum_{q=1}^{5}\mathcal{Q}_q,
\end{equation}
which combined with Cauchy's inequality implies
\begin{equation}\label{id4a}
  \int_{\mathbb{R}_+}\frac{\phi_x^2}{\rho^3}\mathrm{d}x
+\int_0^t\frac{\phi_x^2}{\rho^3}(s,0)\mathrm{d}s
+\int_0^t\int_{\mathbb{R}_+}\frac{\theta\phi_x^2}{\rho^2}
\lesssim
1+\int_{\mathbb{R}_+}\rho\psi^2\mathrm{d}x
+\sum_{q=1}^{5}\int_0^t\int_{\mathbb{R}_+}|\mathcal{Q}_q|,
\end{equation}
where each term $\mathcal{Q}_q$
  on the right-hand side of \eqref{id4} will be defined below.
First, let us define
\begin{equation*}
  \mathcal{Q}_1:=\psi_x^2-\frac{R\phi_x\vartheta_x}{\rho},
\end{equation*}
and by applying Cauchy's inequality, we have
\begin{equation}\label{Q1}
  \int_0^t\int_{\mathbb{R}_+}|\mathcal{Q}_1|
  \lesssim \epsilon \int_0^t\int_{\mathbb{R}_+}\frac{\theta\phi_x^2}{\rho^2}
  +\int_0^t\int_{\mathbb{R}_+}\left[
\psi_x^2+C(\epsilon)\frac{\vartheta_x^2}{\theta}\right].
\end{equation}
Then we consider the term
\begin{equation*}
  \begin{aligned}
    \mathcal{Q}_2:=~& -2\mu\frac{\hat{\rho}_x\phi_x\psi_x}{\rho^3}
    -\frac{\mu\phi_x}{\rho^3}\left(\tilde{u}_{xx}\phi
    +\tilde{\rho}_{xx}\psi\right)
    -\frac{\phi_x}{\rho^2\hat{\rho}}
    (\mu\tilde{u}_{xx}\phi-R\tilde{\rho}_x\hat{\theta}\phi
    -\tilde{u}_x\hat{\rho}\phi\psi)
    \\ &
        -\frac{\phi\psi}{\rho^2}(\tilde{\rho}_x\hat{u}_x+\bar{\rho}_x\tilde{u}_x)
    -\frac{\tilde{\rho}_{xx}\psi^2}{\rho}
    -2\frac{\tilde{\rho}_x\psi\psi_x}{\rho}
    +\frac{\tilde{u}_x\phi\psi_x}{\rho}
    -\frac{R\tilde{\rho}_x\vartheta\phi_x}{\rho^2},
  \end{aligned}
\end{equation*}
which can be controlled as
\begin{equation}\label{Q2a}
  | \mathcal{Q}_2|\lesssim
    m_1^{-3}(\bar{\delta}+\tilde{\delta})|\Psi_x|^2
    +m_1^{-3}|(\tilde{U}_{x},\tilde{U}_{xx})||(\Psi,\Psi_x)|^2.
\end{equation}
In light of \eqref{R4b},
we have
\begin{equation}\label{Q2b}
\begin{split}
  \int_0^t\int_{\mathbb{R}_+}|\mathcal{Q}_2|
    \lesssim
  N^2 m_1^{-4}\tilde{\delta}
  \int_0^t\rho\Phi\left(\frac{\hat{\rho}}{\rho}\right) (s,0)\mathrm{d}s
  + m_1^{-3}(\bar{\delta}+\tilde{\delta})\int_0^t\left\|\Psi_{x}(s)\right\|^2
  \mathrm{d}s.
\end{split}
\end{equation}
For the term
\begin{equation*}
  \begin{aligned}
    \mathcal{Q}_3:=&-\frac{\mu\phi_x}{\rho^3} (\bar{u}_{xx}\phi
    +\bar{\rho}_{xx}\psi)
    -\frac{\phi_x}{\rho^2\hat{\rho}}
    (-R\bar{\rho}_x\hat{\theta}\phi
    -\bar{u}_x\hat{\rho}\phi\psi)\\
    &-2\frac{\bar{\rho}_x\psi\psi_x}{\rho}
    +\frac{\bar{u}_x\phi\psi_x}{\rho}
    -\frac{R\bar{\rho}_x\vartheta\phi_x}{\rho^2},
  \end{aligned}
\end{equation*}
it follows that
\begin{equation*}
  \left|\mathcal{Q}_3\right|
  \lesssim m_1^{-3} \left|(\bar{U}_x,\bar{U}_{xx})\right|
  \left|\Psi\right|\left|\Psi_x\right|
  \lesssim m_1^{-3}
  \|\Psi\|^{\frac{1}{2}}\|\Psi_x\|^{\frac{1}{2}}
  \left|(\bar{U}_x,\bar{U}_{xx})\right|\left|\Psi_x\right|.
\end{equation*}
In view of \eqref{SRW_p3}, \eqref{SRW_p4} and \eqref{apriori},
we apply H\"{o}lder's and Young's inequalities
to deduce
\begin{equation}\label{Q3b}
\begin{aligned}
  \int_0^t\int_{\mathbb{R}_+}\left|\mathcal{Q}_3\right| \lesssim~&
   m_1^{-3}  \int_0^t\|\Psi\|^{\frac{1}{2}}\|(\bar{U}_x,\bar{U}_{xx})\|
  \|\Psi_x\|^{\frac{3}{2}}\mathrm{d}s\\
  \lesssim~& m_1^{-3}N^{\frac{1}{2}}  \bar{\delta}^{\frac{1}{2}}
  \int_0^t(1+s)^{-\frac{1}{2}+\frac{1}{4q}}\left\|\Psi_x(s)\right\|^{\frac{3}{2}}
  \mathrm{d}s
  \\[1mm]
  \lesssim~&  1+m_1^{-4}N^{\frac{2}{3}}
  \bar{\delta}^{\frac{2}{3}}\int_0^t\left\|\Psi_x(s)\right\|^2\mathrm{d}s.
\end{aligned}
\end{equation}
For the term
\begin{equation*}
    \mathcal{Q}_4:=    -\frac{\bar{\rho}_x\bar{u}_x\phi\psi}{\rho^2}
    -\frac{\bar{\rho}_{xx}\psi^2}{\rho},
\end{equation*}
we have
\begin{equation*}
  |\mathcal{Q}_4 |
  \lesssim m_1^{-2}  |\Psi |^2\left[ |\bar{U}_{xx} |+
  |\bar{U}_{x} |^2\right]
  \lesssim m_1^{-2} \|\Psi\|\|\Psi_x\|
  \left[ |\bar{U}_{xx} |+
   |\bar{U}_{x} |^2\right],
\end{equation*}
which combined with \eqref{SRW_p3}-\eqref{SRW_p4} and
\eqref{apriori} yields
\begin{equation}\label{Q4a}
\begin{aligned}
  \int_0^t\int_{\mathbb{R}_+}\left|\mathcal{Q}_4\right|\lesssim~&
   m_1^{-2} \int_0^t\|\Psi\|   \|\Psi_{x}\|
  \left[\|\bar{U}_{xx}\|_{L^1}+
  \|\bar{U}_{x}\|^2\right]\mathrm{d}s\\
  \lesssim~& m_1^{-2}N
  \int_0^t \|\Psi_x(s) \|
  \bar{\delta}^{\frac{1}{4}}
  (1+s)^{-\frac{3}{4}}\mathrm{d}s\\[1mm]
  \lesssim~&  1+m_1^{-4}N^{2}
  \bar{\delta}^{\frac{1}{2}}\int_0^t\|\Psi_x(s)\|^2\mathrm{d}s.
\end{aligned}
\end{equation}
Finally we consider the term
\begin{equation*}
  \begin{aligned}
    \mathcal{Q}_5:=
    -\frac{\mu\phi_x}{\rho^3} \hat{f}_{1x}
    -\frac{\phi_x}{\rho^2\hat{\rho}} (-\mu\bar{u}_{xx}\hat{\rho}
    +\rho\hat{f_2} )
    -\frac{\hat{\rho}_x\psi\hat{f}_1}{\rho^2}
    +\frac{\psi_x\hat{f}_1}{\rho}.
  \end{aligned}
\end{equation*}
H\"{o}lder's inequality gives
\begin{equation}\label{Q5a}
  \int_0^t\int_{\mathbb{R}_+}\left|\mathcal{Q}_5\right|\lesssim
  m_1^{-3}\int_0^t\left[ \|\Psi_x \|
   \| (\bar{u}_{xx},\hat{f_1},\hat{f_2},\hat{f}_{1x} )  \|
  +\|\Psi  \|_{L^{\infty}}
   \|\hat{f_1}  \|_{L^{1}}\right]
  \mathrm{d}s.
\end{equation}
We deduce from  \eqref{f1_hat}
that
\begin{equation} \label{f_hat2}
  |\hat{f}_{1x}|\lesssim
  |\bar{U}_x||\tilde{U}_x|+
  |\tilde{U}_{xx}||\bar{U}-U_m|+|\bar{U}_{xx}||\tilde{U}-U_m|.
\end{equation}
Similar to the derivation of \eqref{point1},
we have that for $q\geq 1$,
\begin{equation}\label{point2}
  \begin{aligned}
         \left\||\bar{U}_x||\tilde{U}_x|+
  |\tilde{U}_{xx}||\bar{U}-U_m|+|\bar{U}_{xx}||\tilde{U}-U_m|\right\|_{L^q}
  \lesssim  \tilde{\delta} \|(\bar{U}_x,\bar{U}_{xx}) \|_{L^{\infty}}.
  \end{aligned}
\end{equation}
Combining the estimates
\eqref{f_hat1}, \eqref{point1}, \eqref{f_hat2} and \eqref{point2}
and utilizing Lemma \ref{lem_SRW},
we deduce
\begin{equation*} %\label{Q5b}
  \big\|\big(\bar{u}_{xx},\hat{f_1},\hat{f_2},\hat{f}_{1x}\big)(s)\big\|
  \lesssim \big\| \bar{u}_{xx}(s)\big\|+
  \tilde{\delta}\big\|(\bar{U}_x,\bar{U}_{xx})(s)\big\|_{L^{\infty}}
  \lesssim \bar{\delta}^{\frac3 7}(1+s)^{-\frac{15}{28}}.
\end{equation*}
Plug this last estimate into \eqref{Q5a} and
use \eqref{f_hat1}-\eqref{point1} again
to have
\begin{equation}\label{Q5c}
  \begin{aligned}
  \int_0^t\int_{\mathbb{R}_+}\left|\mathcal{Q}_5\right|\lesssim~&
  m_1^{-3}\int_0^t\left[\big\|\Psi_x\big\|
  \bar{\delta}^{\frac3 7}(1+s)^{-\frac{15}{28}}
  +\tilde{\delta}(1+s)^{-1}
  \|\Psi\|^{\frac{1}{2}}  \|\Psi_{x}\|^{\frac{1}{2}}
  \right]
  \mathrm{d}s\\
  \lesssim ~&
  1+\left(m_1^{-6}\bar{\delta}^{\frac{6}{7}}
  +m_1^{-12}N^2\tilde{\delta}^4
  \right)\int_0^t\|\Psi_x(s)\|^2\mathrm{d}s.
  \end{aligned}
\end{equation}
Plugging \eqref{id1d}, \eqref{Q1}, \eqref{Q2b}, \eqref{Q3b}, \eqref{Q4a} and
\eqref{Q5c} into \eqref{id4a}, we take $\epsilon$ sufficiently small to find
\begin{equation*}%\label{id4b}
\begin{aligned}
&\int_{\mathbb{R}_+}\frac{\phi_x^2}{\rho^3}\mathrm{d}x
+\int_0^t\frac{\phi_x^2}{\rho^3}(s,0)\mathrm{d}s
+\int_0^t\int_{\mathbb{R}_+}\frac{\theta\phi_x^2}{\rho^2}\\
\lesssim~&1
+\int_0^t\int_{\mathbb{R}_+}\left[
\psi_x^2+\frac{\vartheta_x^2}{\theta}\right]
+m_{1}^{-12}m_2^{-12}N^{16}
  \left[\bar{\delta}^{\frac{1}{2}}+\tilde{\delta}\right]
  \int_0^t\|\phi_x(s)\|^2\mathrm{d}s\\
&+m_1^{-12} N^6\left[\bar{\delta}^{\frac{1}{2}}+\tilde{\delta}\right]
\left[\int_0^t\rho\Phi\left(\frac{\hat{\rho}}{\rho}\right)
  (s,0)\mathrm{d}s+\int_0^t\int_{\mathbb{R}_+}\left(\frac{\psi_x^2}{\theta}
  +\frac{\vartheta_x^2}{\theta^2}\right)\right],
\end{aligned}
\end{equation*}
 which combined with \eqref{id1d} gives
\begin{equation*}
\begin{aligned}
  &\int_{\mathbb{R}_+}\frac{\phi_x^2}{\rho^3}\mathrm{d}x
+\int_0^t\frac{\phi_x^2}{\rho^3}(s,0)\mathrm{d}s
+\int_0^t\int_{\mathbb{R}_+}\frac{\theta\phi_x^2}{\rho^2}
\\ &\quad
\lesssim 1
+\int_0^t\int_{\mathbb{R}_+}\left[
\psi_x^2+\frac{\vartheta_x^2}{\theta}\right]
+
  m_{1}^{-50}m_2^{-50}N^{50}
 \left[\bar{\delta}^{\frac{1}{2}}+\tilde{\delta}\right] \int_0^t\int_{\mathbb{R}_+}\frac{\theta\phi_x^2}{\rho^2}.
\end{aligned}
\end{equation*}
We take $\epsilon_0>0$ small enough and use \eqref{asmp1} to have
\begin{equation}\label{id4e}
  \int_{\mathbb{R}_+}\frac{\phi_x^2}{\rho^3}\mathrm{d}x
+\int_0^t\frac{\phi_x^2}{\rho^3}(s,0)\mathrm{d}s
+\int_0^t\int_{\mathbb{R}_+}\frac{\theta\phi_x^2}{\rho^2}
\lesssim
1+\int_0^t\int_{\mathbb{R}_+}\left[
\psi_x^2+\frac{\vartheta_x^2}{\theta}\right].
\end{equation}
The estimate \eqref{est1} follows by plugging \eqref{id4e}
into \eqref{id1d}
and using the condition \eqref{asmp1}
for a sufficiently small $\epsilon_0>0$.
Combine \eqref{est1} and \eqref{id4e} to deduce \eqref{est2}.
The proof of the lemma is completed.
\end{proof}

\subsection{Uniform bounds on density} \label{sec_den}
Having obtained the energy estimate \eqref{est1},
we can proceed to deduce the positive lower and upper bounds
of the density $\rho(t,x)$
uniformly in time $t$ and space $x$ in this subsection.
For this purpose, we transform the outflow problem into
the corresponding problem in the Lagrangian coordinate
by introducing the Lagrangian variable
\begin{equation}\label{y}
  y=-u_-\int_{0}^{t}\rho(s,0)\mathrm{d}s+\int_{0}^{x}\rho(t,z)\mathrm{d}z.
\end{equation}
By the coordinate change $(t,x)\mapsto(t,y)$, the domain
$[0,T]\times\mathbb{R}_+$ is mapped into
$$\Omega_T:=\{(t,y):0\leq t\leq T, y>Y(t)\}\quad {\rm with}\quad
Y(t):=-u_-\int_{0}^{t}\rho(s,0)\mathrm{d}s,$$
and the outflow problem \eqref{NS_E}-\eqref{bdy}
is transformed into the following initial boundary value problem
\begin{equation}\label{NS_L}
\left\{
  \begin{aligned}
    v_{t}-u_y&=0,\\[2mm]
    u_{t}+P_y&=\left(\frac{\mu u_y}{v}\right)_y, \\
    \left(c_v\theta+\frac{u^2}{2}\right)_{t}+(Pu)_y
    &=\left(\frac{\kappa\theta_y}{v}
    +\frac{\mu uu_y}{v}\right)_y \quad {\rm for}\
    y>Y(t),\\
    (u,\theta)|_{y=Y(t)}&=(u_-,\theta_-),\\[3mm]
    (v,u,\theta)|_{t=0}&=(v_0,u_0,\theta_0).
  \end{aligned}\right.
\end{equation}
Here  $v={1}/{\rho}$ stands for the specific volume of the gas
and $v_0=1/\rho_0$.
The basic energy estimate \eqref{est1} in Eulerian coordinate
can be rewritten as a corresponding estimate in Lagrangian coordinate
as a direct consequence of the transformation \eqref{y}.
\begin{corollary}
Suppose that the conditions listed in Lemma \ref{lem_bas} hold.
Then
\begin{equation}\label{est1a}
  \sup_{0\leq t\leq T}\int_{Y(t)}^{\infty}{\mathcal{E}}\mathrm{d}y
    +\int_0^{T}\int_{Y(t)}^{\infty}\left[\frac{\psi_y^2}{v \theta}
  +\frac{\vartheta_y^2}{v\theta^2}\right]\mathrm{d}y\mathrm{d}t
  \lesssim 1.
\end{equation}
\end{corollary}
Note that the function $Y(t)$ describing the boundary in the Lagrangian coordinate
is part of the unknown, that is, the problem \eqref{NS_L} is a free boundary problem.
To obtain the uniform bounds of the specific volume $v$
for the free boundary problem \eqref{NS_L},
we introduce
the time-dependent domain $\Omega_{i}(t)$
with $i\in\mathbb{Z}$ and $t\in[0,T]$
as
\begin{equation} \label{Omegai}
  \Omega_{i}(t):=\left\{
  \begin{aligned}
    &[Y(t),[Y(t)]+2]\quad &&{\rm if}\ i=[Y(t)]+1,\\
    &[i,i+1]\quad &&{\rm else.}
  \end{aligned}\right.
\end{equation}
Based on  the basic energy estimate \eqref{est1a}, we have the following lemma.
 \begin{lemma} \label{lem_v1}
Suppose that the conditions listed in Lemma \ref{lem_bas} hold.
 Then there exists a positive constant $C_0$, depending solely on
   $\inf_{x\in\mathbb{R}_+}\{\rho_0(x),\theta_0(x)\}$
  and $\|(\phi_0,\psi_0,\vartheta_0)\|_1$, such that
 for
 all pair $(s,t)$ with $0\leq s\leq t\leq T$
and  integer $i\geq[Y(t)]+1$,
 \begin{equation}\label{bound1}
   C_0^{-1}\leq \int_{\Omega_i(t)}{v}(s,y)\mathrm{d}y\leq C_0,\quad
    C_0^{-1}\leq \int_{\Omega_i(t)}{\theta}(s,y)\mathrm{d}y\leq C_0,
 \end{equation}
 and there are points $a_i(s,t),b_i(s,t)\in\Omega_i(t)$
    satisfying
 \begin{equation}\label{bound2}
   C_0^{-1}\leq v(s,a_i(s,t))\leq C_0,\quad
   C_0^{-1}\leq \theta(s,b_i(s,t))\leq C_0.
 \end{equation}
  \end{lemma}
\begin{proof}
  Let $0\leq s\leq t\leq T$ and $i\geq[Y(t)]+1$.
  According to the definition of $Y(t)$ and the sign of $u_-$,
  we have $Y(s)\leq Y(t)$ and $\Omega_{i}(t)\subset [Y(s),\infty).$
  In view of \eqref{est1a}, we get
  \begin{equation*}
    \int_{\Omega_i(t)}\Phi\left(\frac{v}{\hat{v}}\right)(s,y)\mathrm{d}y+
    \int_{\Omega_i(t)}\Phi\left(\frac{\theta}{\hat{\theta}}\right)(s,y)\mathrm{d}y
    \lesssim 1.
  \end{equation*}
  Apply Jensen's inequality to
the convex function $\Phi$ to obtain
  \begin{equation*}
     \Phi\left(\frac{1}{|\Omega_i(t)|}\int_{\Omega_i(t)}
     \frac{v}{\hat{v}}(s,y)\mathrm{d}y\right)+
    \Phi\left(\frac{1}{|\Omega_i(t)|}\int_{\Omega_i(t)}\frac{\theta}{\hat{\theta}}
    (s,y)\mathrm{d}y \right)\leq C.
  \end{equation*}
  Let $\alpha$ and $\beta$ be the two positive roots of the equation $\Phi(z)=C$.
  Then we have
  \begin{equation*}
    \alpha\leq
    \frac{1}{|\Omega_i(t)|}\int_{\Omega_i(t)}\frac{v}{\hat{v}}(s,y)\mathrm{d}y\leq \beta,
    \quad
    \alpha\leq\frac{1}{|\Omega_i(t)|}\int_{\Omega_i(t)}\frac{\theta}{\hat{\theta}}
    (s,y)\mathrm{d}y
    \leq \beta.
  \end{equation*}
  These estimates imply \eqref{bound1}.
  Finally we employ the mean value theorem
 to \eqref{bound1} to find $a_i(s,t),b_i(s,t)\in\Omega_i(t)$ satisfying
  \eqref{bound2}.
  The proof of the lemma is completed.
\end{proof}
We deduce a local representation of the solution
$v$ for the free boundary problem
\eqref{NS_L}
in the next lemma
by modifying Jiang's argument  for
fixed domains in  \cite{J99MR1671920,J02MR1912419}.
To this end, we introduce the cutoff function
  $\varphi_z\in W^{1,\infty}(\mathbb{R})$ with
parameter $z\in\mathbb{R}$ by
  \begin{equation} \label{varphi}
  \varphi_z(y)=
  \begin{cases}
    1,\quad& y<[z]+4,\\
    [z]+5-y,\quad& [z]+4\leq y< [z]+5,\\
    0,\quad& y\geq [z]+5.
  \end{cases}
\end{equation}
\begin{lemma}
  Let $(\tau,z)\in\Omega_T$. Then
\begin{equation}\label{v_form}
  v(t,y)=B_z(t,y)A_z(t)+\frac{R}{\mu}\int_0^t
  \frac{B_z(t,y)A_z(t)}{B_z(s,y)A_z(s)}\theta(s,y)\mathrm{d}s
\end{equation}
for all $t\in[0,\tau]$ and $y\in
I_z(\tau):=(Y(\tau),\infty)\cap([z]-1,[z]+4)$, where
  \begin{align}   \label{B_z}
    B_z(t,y)&:=v_0(y)\exp\left\{\frac{1}{\mu}
    \int_{y}^{\infty}
        \left(u_0(\xi)-u(t,\xi)\right)\varphi_z(\xi)\mathrm{d}\xi\right\},\\
        \label{A_z}
    A_z(t)&:=\exp\left\{\frac{1}{\mu}\int_0^t
    \int_{[z]+4}^{[z]+5}\left(\frac{\mu u_y}{v}-P\right)\mathrm{d}\xi\mathrm{d}s\right\}.
  \end{align}
\end{lemma}
\begin{proof}
We  multiply $\eqref{NS_L}_2$ by $\varphi_z$ to get
 \begin{equation} \label{id5}
   (\varphi_z u)_t=\left[\left(\mu\frac{u_y}{v}-P\right)\varphi_z\right]_y
   -\varphi_z'\left(\mu\frac{u_y}{v}-P\right).
 \end{equation}
Let $(t,y)\in[0,\tau]\times I_z(\tau)$.
Since  $y> %Y(\tau)\geq
Y(s)$ for each $s\in [0,\tau]$, we have
$[0,\tau]\times[y,\infty)\subset\Omega_T$.
In light of the identity $\varphi_z(y)=1$ and \eqref{NS_L}$_1$,
we integrate \eqref{id5} over
$[0,t]\times[y,\infty)$ to get
 \begin{equation*}
   -\int_y^{\infty}\varphi_z(\xi)(u(t,\xi)-u_0(\xi))\mathrm{d}\xi
   =\mu\ln\frac{v(t,y)}{v_0(y)}-R\int_0^{t}\frac{\theta(s,y)}{v(s,y)}\mathrm{d}s
   +\int_0^{t}\int_{[z]+4}^{[z]+5}\left(P-\mu\frac{u_y}{v}\right).
 \end{equation*}
 This implies that for each $t\in[0,\tau]$,
 \begin{equation}\label{id6}
  \frac{1}{v(t,y)}\exp\left\{\frac{R}{\mu}\int_0^t\frac{\theta(s,y)}{v(s,y)}\mathrm{d}s\right\}
  =\frac{1}{B_z(t,y)A_z(t)}.
\end{equation}
Multiplying \eqref{id6} by $R\theta(t,y)/\mu$
and integrating the resulting identity over $[0, t]$, we have
\begin{equation*}
  \exp\left\{\frac{R}{\mu}\int_0^{t}\frac{\theta(s,y)}{v(s,y)}\mathrm{d}s\right\}
  =1+\frac{R}{\mu}\int_0^{t} \frac{\theta(s,y)}{B_z(s,y)A_z(s)}\mathrm{d}s.
\end{equation*}
We then plug this identity into \eqref{id6} to obtain
\eqref{v_form} and complete the proof of the lemma.
\end{proof}
The following lemma
is devoted to showing
 the bounds of the specific volume $v(\tau,z)$
 uniformly in the time $\tau$ and the Lagrangian variable $z$.
\begin{lemma} \label{lem_boun}
Suppose that the conditions listed in Lemma \ref{lem_bas} hold.
Then
     \begin{equation}\label{est3}
   C_1^{-1}\leq v(\tau,z)\leq C_1\quad {\rm for\ all}\ (\tau,z)\in\Omega_T.
  \end{equation}
\end{lemma}
\begin{proof}
Let $(\tau,z)\in\Omega_T$ be arbitrary but fixed.
The proof is divided into three steps.

\noindent{\bf Step 1}.
It follows from  Cauchy's inequality and \eqref{est1a} that
%\begin{equation*}
%  \left|\int_{z}^{\infty}\left(u_0(y)-u(t,y)\right)\varphi_z(y)\mathrm{d}y\right|
%  \lesssim\left[\int_z^{[z]+5}u^2\right]^{\frac{1}{2}}
%  +\left[\int_z^{[z]+5}u_0^2\right]^{\frac{1}{2}}\lesssim 1,
%\end{equation*}
%and hence
\begin{equation}\label{est3.1}
  B_z(t,y)\sim1 \quad {\rm for\ all}\ (t,y)\in[0,\tau]\times I_z(\tau).
\end{equation}
Let $0\leq s\leq t\leq \tau$.
For each $0\leq t' \leq t$, there exists $y(t')\in\left[[z]+4,[z]+5\right]$
such that
\begin{equation*} \label{pro12}
  \theta(t',y(t' ))=\inf_{([z]+4,[z]+5)}\theta(t' ,\cdot),
\end{equation*}
Apply Cauchy's inequality to have
\begin{equation*}
  \frac{\mu u_y}{v}-P\leq
  \frac{Cu_y^2}{v\theta}-\frac{R\theta}{2v}
  \leq \frac{C\psi_y^2}{v\theta}
        +\frac{C\hat{u}_y^2}{v\theta}-\frac{R\theta}{2v}.
\end{equation*}
In view of \eqref{decay1}, \eqref{SRW_p3},
\eqref{est1a}, \eqref{asmp1} and \eqref{bound1},
we apply Jensen's inequality
for the convex function $1/x$ to deduce
\begin{equation}\label{est3.2}
    \begin{aligned}
      &\int_s^{t}
    \int_{[z]+4}^{[z]+5}\left[\frac{\mu u_y}{v}-P\right]  \\
       &\quad \leq C+C N   m_2^{-1}(\tilde{\delta}+\bar{\delta}) ({t}-s)
         -\frac{R}{2}\int_s^{t}  \theta(t',y(t'))
        \int_{[z]+4}^{[z]+5}v^{-1}(t',y)\mathrm{d}y\mathrm{d}t'    \\
       &\quad \leq C+C N   m_2^{-1}(\tilde{\delta}+\bar{\delta}) ({t}-s)
         -\frac{R}{2}\int_s^{t}  \theta(t',y(t'))
        \left[\int_{[z]+4}^{[z]+5}v\mathrm{d}y\right]^{-1}\mathrm{d}t'\\
      &\quad  \leq  C+C\epsilon_0 (t-s)
        -C^{-1}\int_s^{t} \theta(t',y(t')) \mathrm{d}t',
    \end{aligned}
  \end{equation}
Since $[z]+4\geq [Y(t')]+2$, we derive
from  \eqref{Omegai} that
$
  \Omega_{[z]+4}(t')=\left[[z]+4,[z]+5\right].
$
We then apply H\"{o}lder's and Cauchy's inequalities to obtain from
\eqref{decay1}, Lemma \ref{lem_v1} and \eqref{est1a} that
\begin{equation}\label{est3.3}
\begin{split}
  &\left|\int_s^{t}\int_{b_{[z]+4}(t',t')}^{y(t')}
\frac{\theta_y}{\theta}(t',\xi)\mathrm{d}\xi\mathrm{d}t'
\right|\\ &\quad
   \leq
\int_s^{t}\int_{\Omega_{[z]+4}(t')}
\left|\frac{\hat\theta_y}{\theta}(t',\xi)+
\frac{\vartheta_y}{\theta}(t' ,\xi)\right|\mathrm{d}\xi\mathrm{d}t'  \\ &\quad
 \leq
m_2^{-1}(\tilde{\delta}+\bar{\delta})({t}-s)
+\int_s^{t}\left|\int_{Y(t')}^{\infty}
\frac{\vartheta_y^2}{v\theta^2}(t',\xi)\mathrm{d}\xi \right|^{\frac{1}{2}}
\left|\int_{\Omega_{[z]+4}(t')} v (t',\xi)\mathrm{d}\xi\right|
^{\frac{1}{2}}\mathrm{d}t'\\ &\quad
 \leq  C (t-s) +C \int_s^t\int_{Y(t')}^{\infty}
\frac{\vartheta_y^2}{v\theta^2}(t',\xi)\mathrm{d}\xi\mathrm{d}t'   \\ &\quad
 \leq  C (t-s) +C.
\end{split}
\end{equation}
Applying Jensen's inequality to the convex function $e^{x}$,
we have from %\eqref{pro12},
\eqref{bound2} and \eqref{est3.3} that
\begin{equation*}\label{pro11}
  \begin{split}
    &\int_s^{t}\theta(t' ,y(t' ))\mathrm{d}t'
    =\int_s^{t}\exp\left(\ln\theta(t' ,y(t' ))\right)\mathrm{d}t' \\
    &\quad\geq  ({t}-s)\exp\left(\frac{1}{{t}-s}\int_s^{t}
    \ln\theta(t' ,y(t ' )) \mathrm{d}t' \right)\\
    &\quad\geq  ({t}-s)\exp\left(\frac{1}{{t}-s}\int_s^{t}
\left[\int_{b_{[z]+4}(t ',t' )}^{y(t'  )}
\frac{\theta_y}{\theta}(t' ,\xi)\mathrm{d}\xi+
\ln\theta(t' ,b_{[z]+4}(t' ,t '))\right]\mathrm{d}t'\right)\\
&\quad\geq  ({t}-s)\exp\left(-\ln C_0
-\frac{1}{{t}-s}\left|\int_s^{t}\int_{b_{[z]+4}(t ',t' )}^{y(t' )}
\frac{\theta_y}{\theta}(t ',\xi)\mathrm{d}\xi \mathrm{d}t'
\right|\right)\\
&\quad\geq  \frac{{t}-s}{C}\exp\left({-\frac{C}{ {t}-s }}\right).
  \end{split}
\end{equation*}
This implies
\begin{equation}\label{est3.4}
  -\int_s^{t} \theta(t' ,y(t' ))\mathrm{d}t'
  \leq\left\{
  \begin{aligned}
    &0  \quad& &{\rm if }\ 0\leq {t}-s\leq 1,\\
    &-C^{-1}({t}-s)\quad& &{\rm if }\ {t}-s\geq 1.
  \end{aligned}\right.
\end{equation}
Plugging \eqref{est3.4} into \eqref{est3.2} and taking $\epsilon_0>0$ small enough,
we have  for each $s\in[0,{t}]$ that
\begin{equation*}
 \int_s^{t}
    \int_{[z]+4}^{[z]+5}\left[\frac{\mu u_y}{v}-P\right]
    \leq C-C^{-1}({t}-s).
\end{equation*}
According to the definition \eqref{A_z}, we then obtain
\begin{equation}\label{est3.5}
    0\leq A_z({t})\leq C\mathrm{e}^{-{t}/C},\quad
    \frac{A_z(t)}{A_z(s)}\leq C\mathrm{e}^{-({t}-s)/{C}}
    \quad { \rm  for\ all}\ 0\leq s\leq {t}\leq \tau.
  \end{equation}

  \noindent{\bf Step 2}.
   Plugging \eqref{est3.1} and \eqref{est3.5} into \eqref{v_form}, we infer
   that for all $(t,y)\in [0,\tau]\times I_z(\tau)$,
   \begin{equation}\label{est3.6}
     \int_0^t\frac{A_z(t)}{A_z(s)}\theta(s,y)\mathrm{d}s\lesssim
     v(t,y)\lesssim 1+\int_0^t\theta(s,y)\mathrm{e}^{-\frac{t-s}C}\mathrm{d}s.
   \end{equation}
In light of the fundamental theorem of calculus, we deduce
from  \eqref{decay1} and \eqref{bound1} that for
 $y\in I_z(\tau)$ and $0\leq s\leq t\leq \tau$,
     \begin{equation} \label{est3.7a}
  \begin{aligned}
    &\left|\theta(s,y)^{\frac{1}{2}}-\theta(s,b_{[z]+2}(s,\tau))^{\frac{1}{2}}
    \right| \\ &\quad
%    \lesssim
%     \int_{I_z(\tau)}\theta^{-\frac12}|\theta_y|(s,\xi)\mathrm{d}\xi\\ &\quad
    \lesssim  \int_{I_z(\tau)}\theta^{-\frac12}\left|\hat{\theta}_y\right| (s,\xi)\mathrm{d}\xi+
    \int_{I_z(\tau)}\theta^{-\frac12}|\vartheta_y|(s,\xi)\mathrm{d}\xi
    \\ &\quad
    \lesssim  m_2^{-\frac{1}{2}}(\tilde{\delta}+\bar{\delta})+
    \left[\int_{I_z(\tau)}\frac{\vartheta_y^2}{v\theta^2}
    (s,\xi)\mathrm{d}\xi\right]^{\frac{1}{2}}
    \left[\int_{I_z(\tau)}{v\theta}(s,\xi)\mathrm{d}\xi\right]^{\frac{1}{2}}
    \\ &\quad
    \lesssim   m_2^{-\frac{1}{2}}(\tilde{\delta}+\bar{\delta})+
    \sup_{I_z(\tau)}v^{\frac12}(s,\cdot)
    \left[\int_{I_z(\tau)}\frac{\vartheta_y^2}{v\theta^2}(s,\xi)
    \mathrm{d}\xi\right]^{\frac{1}{2}},
  \end{aligned}
  \end{equation}
  where we have used
  $b_{[z]+2}(s,\tau)\in \Omega_{[z]+2}(\tau)\subset I_z(\tau)$.
   Combine \eqref{est3.7a} with \eqref{bound2} and \eqref{asmp1} to give
  \begin{equation}\label{est3.7b}
    1-C
    \sup_{I_z(\tau)}v(s,\cdot)
    \int_{I_z(\tau)}\frac{\vartheta_y^2}{v\theta^2}(s,\xi)\mathrm{d}\xi
    \lesssim \theta(s,y)\lesssim
    1+\sup_{I_z(\tau)}v(s,\cdot)
    \int_{I_z(\tau)}\frac{\vartheta_y^2}{v\theta^2}(s,\xi)\mathrm{d}\xi.
  \end{equation}
  We plug \eqref{est3.7b} into \eqref{est3.6} to obtain
  \begin{equation*}
  \begin{aligned}
    v(t,y)\lesssim 1+\int_0^t\sup_{I_z(\tau)}v(s,\cdot)\int_{I_z(\tau)}
    \frac{\vartheta_y^2}{v\theta^2}(s,\xi)\mathrm{d}\xi\mathrm{d}s.
  \end{aligned}
  \end{equation*}
  Taking the supremum over $I_z(\tau)$ with respect to $y$, we have
 \begin{equation} \label{est3.8a}
  \sup_{I_z(\tau)}v(t,\cdot)\lesssim 1+\int_0^t
  \sup_{I_z(\tau)}v(s,\cdot)\int_{\Omega_i(\tau)}
    \frac{\vartheta_y^2}{v\theta^2}(s,\xi)\mathrm{d}\xi\mathrm{d}s.
 \end{equation}
 Applying Gronwall's inequality to \eqref{est3.8a},
 we can deduce from \eqref{est1a}
 that
 \begin{equation}\label{est3.8b}
  \sup_{I_z(\tau)} v (t,\cdot)\leq C_1\quad { \rm  for\ all}\  t\in[0,\tau],
 \end{equation}
 where
 $C_1>0$ is some constant independent of $t$, $\tau$ and $z$.
 Noting that
 $z\in I_z(\tau)$, we deduce from \eqref{est3.8b} that
 $v(\tau,z)\leq C_1$. Since
 $(\tau,z)\in \Omega_T$ is arbitrary, we conclude
 \begin{equation}\label{est3.8c}
      v (\tau,z)\leq C_1 \quad { \rm  for\ all}\  (\tau,z)\in\Omega_T.
  \end{equation}

\noindent{\bf Step 3}.
On the other hand, in view of \eqref{bound1}, \eqref{est3.1} and \eqref{est3.5},
we integrate \eqref{v_form} on $I_z(\tau)$
with respect to $y$ to find
  \begin{equation*}
    \begin{aligned}
      1\lesssim\int_{I_z(\tau)}v(t,y)\mathrm{d}y\lesssim \mathrm{e}^{-t/C}
      +\int_0^t\frac{A_z(t)}{A_z(s)}\mathrm{d}s.
    \end{aligned}
  \end{equation*}
Consequently, we have
\begin{equation}\label{est3.9}
  \int_0^t\frac{A_z(t)}{A_z(s)}\mathrm{d}s\gtrsim 1-C\mathrm{e}^{-{t}/{C }}.
\end{equation}
Inserting \eqref{est3.7b}, \eqref{est3.8c} and
\eqref{est3.9} into \eqref{est3.6}, we have
  \begin{equation}\label{est3.10a}
  \begin{aligned}
    v(t,y)
    \gtrsim~& \int_{0}^t\frac{A_z(t)}{A_z(s)}\mathrm{d}s
    -C \int_{0}^t\frac{A_z(t)}{A_z(s)}\int_{I_z(\tau)}
\frac{\vartheta_y^2}{v\theta^2}\mathrm{d}\xi\mathrm{d}s\\
    \gtrsim~& 1-C \mathrm{e}^{-{t}/{C }}-
    C \left(\int_{0}^{{t}/2}+\int_{{t}/2}^t\right)
    \frac{A_z(t)}{A_z(s)}\int_{I_z(\tau)}
    \frac{\vartheta_y^2}{v\theta^2}\mathrm{d}\xi\mathrm{d}s
    \\
    \gtrsim~& 1-C \mathrm{e}^{-t/C }
    -C \int_{0}^{{t}/2}\mathrm{e}^{-\frac{t-s}{C}}
    \int_{I_z(\tau)}\frac{\vartheta_y^2}{v\theta^2}\mathrm{d}\xi\mathrm{d}s
    -C \int_{{t}/2}^t\int_{I_z(\tau)}\frac{\vartheta_y^2}{v\theta^2}\\
    \gtrsim~& 1-C \mathrm{e}^{-t/C }
    -C \mathrm{e}^{-\frac{t}{2C}}
    -C \int_{{t}/2}^t\int_{I_z(\tau)}\frac{\vartheta_y^2}{v\theta^2}\\
    \gtrsim~& 1  \quad {\rm for\ all}\ (t,y)\in [T_0,\tau]\times I_z(\tau),
  \end{aligned}
\end{equation}
where $T_0$ is a positive constant independent of $t$.
In particular, the estimate \eqref{est3.10a} implies
\begin{equation}\label{est3.10b}
  v(\tau,z)\gtrsim 1 \quad {\rm for\ all}\ \tau\geq T_0,\ z>Y(\tau).
\end{equation}
As in \cite{K82MR651877,KS77MR0468593}, we can derive a positive lower bound for $v$,
that is,
  \begin{equation}\label{est3.10c}
    v(\tau,z)\gtrsim \mathrm{e}^{ -Ct}  \quad {\rm for}\ (\tau,z) \in \Omega_T.
  \end{equation}
Finally, we combine \eqref{est3.10c},
\eqref{est3.8c} and \eqref{est3.10b} to get \eqref{est3}.
This completes the proof.
\end{proof}

As a corollary of Lemma \ref{lem_boun}, we obtain the bounds for the density
$\rho(t,x)$ uniformly in time $t$ and space $x$.
\begin{corollary}
Suppose that the conditions listed in Lemma \ref{lem_bas} hold.  Then
  \begin{equation}\label{bound_rho}
   C_1^{-1}\leq \rho(t,x)\leq C_1\quad {\rm for\ all}\ (t,x)\in[0,T]\times\mathbb{R}_+,
  \end{equation}
   where the positive constant $C_1$ depends solely
on $\inf_{x\in\mathbb{R}_+}\{\rho_0(x),\theta_0(x)\}$
  and $\|(\phi_0,\psi_0,\vartheta_0)\|_1$.
\end{corollary}

\subsection{Uniform estimates for the perturbation} \label{sec_norm}
In this subsection, we will estimate the $H^1_x$-norm of
the perturbation $(\phi,\vartheta,\psi)(t,x)$ uniformly in time $t$.
First we can get the following uniform $L^2$-norm estimate.
\begin{lemma}
  \label{lem_the1}
Suppose that the conditions listed in Lemma \ref{lem_bas} hold.  Then
\begin{equation}\label{est4}
    \sup_{0\leq t\leq T} \left\|(\phi,\vartheta,\psi)(t)\right\|^2
    +\int_0^T\int_{\mathbb{R}_+}\left[(1+\theta+\psi^2)\psi_x^2+\vartheta_x^2\right]
    \mathrm{d}x\mathrm{d}t
    \lesssim 1.
  \end{equation}
\end{lemma}
\begin{proof}
  We divide the proof into five steps.

  \noindent{\bf Step 1}. First, for each $t\geq 0$ and $a>0$, we denote
  $$
  \Omega'_a(t):=\{x\in\mathbb{R}_+: \vartheta(t,x)>a\}.
  $$
  Then it follows from \eqref{est1} and \eqref{bound_rho} that
  \begin{equation}\label{est4.1}
  \begin{aligned}
    &\sup_{0\leq t\leq T}\left[\int_{\mathbb{R}_+}\phi^2\mathrm{d}x
    +\int_{\mathbb{R}_+\backslash\Omega'_a(t)}\vartheta^2\mathrm{d}x
    +\int_0^t|\phi(s,0)|^2 \mathrm{d}s\right]
    +\int_0^T\int_{\mathbb{R}_+\backslash\Omega'_a(t)}
    \left[\psi_x^2+\vartheta_x^2\right]\\[0.5mm]
      \leq~& C(a)\sup_{0\leq t\leq T}\left[\int_{\mathbb{R}_+}\rho\mathcal{E}\mathrm{d}x
    +\int_0^t\rho\Phi\left(\frac{\hat{\rho}}{\rho}\right)(s,0)\mathrm{d}s \right]
    +C(a)\int_0^T\int_{\mathbb{R}_+}\left[\frac{\psi_x^2}{\theta}
    +\frac{\vartheta_x^2}{\theta^2}\right]\leq C(a).
    \end{aligned}
  \end{equation}

  \noindent{\bf Step 2}. We now estimate the integral $\int_0^T\int_{\Omega'_a(t)}\vartheta_x^2$.
  To this end, we multiply $\eqref{per1}_3$
  by $(\vartheta-2)_+:=\max\{\vartheta-2,0\}$ and integrate the resulting identity
  over $(0,t)\times \mathbb{R}_+$ to obtain
  \begin{equation}\label{id4.1}
    \begin{aligned}
    &\frac{c_v}{2}\int_{\mathbb{R}_+}\rho(\vartheta-2)_+^2\mathrm{d}x
    +\kappa\int_0^t\int_{\Omega'_2(s)}\vartheta_x^2
    +R\int_0^t\int_{\mathbb{R}_+}\rho\theta\psi_x(\vartheta-2)_+\\[0.5mm]
   &\quad =\frac{c_v}{2}\int_{\mathbb{R}_+}\rho_0(\vartheta_0-2)_+^2\mathrm{d}x
    +\int_0^t\int_{\mathbb{R}_+}f_3(\vartheta-2)_+
    +\mu\int_0^t\int_{\mathbb{R}_+}\psi_x^2(\vartheta-2)_+.
  \end{aligned}
  \end{equation}
  To estimate the last term of \eqref{id4.1}, we multiply $\eqref{per1}_2$
  by $2\psi(\vartheta-2)_+$ and integrate the resulting identity
  over $(0,t)\times\mathbb{R}_+$ to find
  \begin{equation}\label{id4.2}
    \begin{aligned}
      &\int_{\mathbb{R}_+}\psi^2\rho(\vartheta-2)_+\mathrm{d}x
      +2\mu\int_0^t\int_{\mathbb{R}_+}\psi_x^2(\vartheta-2)_+
      -\int_0^t\int_{\Omega'_2(s)}\rho\psi^2(\vartheta_t+u\vartheta_x)\\[0.5mm]
      &\quad =\int_{\mathbb{R}_+}\psi_0^2\rho_0(\vartheta_0-2)_+\mathrm{d}x
      +2R\int_0^t\int_{\mathbb{R}_+}\rho\theta\psi_x(\vartheta-2)_+
      +2R\int_0^t\int_{\Omega'_2(s)}\rho\theta\psi\vartheta_x\\[0.5mm]
      &\qquad +2R\int_0^t\int_{\mathbb{R}_+}\psi(\hat{\rho}
      \hat{\theta})_x(\vartheta-2)_+
      -2\mu\int_0^t\int_{\Omega'_2(s)}\psi\psi_x\vartheta_x
      +2\int_0^t\int_{\mathbb{R}_+}f_2\psi(\vartheta-2)_+.
    \end{aligned}
  \end{equation}
 Combining \eqref{id4.2} and \eqref{id4.1}, we have from $\eqref{per1}_3$
 that
 \begin{equation}\label{id4.3}
    \begin{aligned}
      &\int_{\mathbb{R}_+}\left[\frac{c_v}{2}\rho(\vartheta-2)_+^2
      +\psi^2\rho(\vartheta-2)_+\right]\mathrm{d}x
      +\kappa\int_0^t\int_{\Omega'_2(s)}\vartheta_x^2
      +\mu\int_0^t\int_{\mathbb{R}_+}\psi_x^2(\vartheta-2)_+\\
      &\quad=\int_{\mathbb{R}_+}\left[\frac{c_v}{2}\rho_0(\vartheta_0-2)_+^2
      +\psi_0^2\rho_0(\vartheta_0-2)_+\right]\mathrm{d}x
      +\sum_{p=1}^6 \mathcal{J}_p,
    \end{aligned}
  \end{equation}
where each term $\mathcal{J}_p$ in the decomposition will be defined below.
We now define and estimate all the terms in the decomposition.
We first consider
\begin{equation*}
  \mathcal{J}_1:=
  R\int_0^t\int_{\mathbb{R}_+}\rho\theta\psi_x(\vartheta-2)_+
\quad {\rm and}\quad
  \mathcal{J}_2:=
  2R\int_0^t\int_{\Omega'_2(s)}\rho\theta\psi\vartheta_x.
\end{equation*}
% \begin{equation}\label{id4.3a}
%    \begin{aligned}
%      &\int_{\mathbb{R}_+}\left[\frac{c_v}{2}\rho(\vartheta-2)_+^2
%      +\psi^2\rho(\vartheta-2)_+\right]
%      +\kappa\int_0^t\int_{\Omega'_2(s)}\vartheta_x^2
%      +\mu\int_0^t\int_{\mathbb{R}_+}\psi_x^2(\vartheta-2)_+\\
%      =&\int_{\mathbb{R}_+}\left[\frac{c_v}{2}\rho_0(\vartheta_0-2)_+^2
%      +\psi_0^2\rho_0(\vartheta_0-2)_+\right]
%      +\underbrace{R\int_0^t\int_{\mathbb{R}_+}\rho\theta\psi_x(\vartheta-2)_+}_{I_1}\\
%      &+\underbrace{2R\int_0^t\int_{\Omega'_2(s)}\rho\theta\psi\vartheta_x}_{I_2}
%      +\underbrace{\int_0^t\int_{\Omega'_2(s)}\left[h(\vartheta-2)_+
%      +c_v^{-1}h\psi^2
%      +2g\psi(\vartheta-2)_+\right]}_{I_3}\\
%      &+\underbrace{2R\int_0^t\int_{\mathbb{R}_+}
%      \psi(\tilde{\rho}\tilde{\theta})_x(\vartheta-2)_+}_{I_4}\\
%      ~&+\underbrace{\int_0^t\int_{\Omega'_2(s)}\left[\mu\psi_x^2c_v^{-1}\psi^2
%      -P\psi_xc_v^{-1}\psi^2-2\mu\psi\psi_x\vartheta_x\right]}_{I_5}
%      +\underbrace{\int_0^t\int_{\Omega'_2(s)}c_v^{-1}\kappa\psi^2\vartheta_{xx}}_{I_6}.
%    \end{aligned}
%  \end{equation}
  In light of \eqref{bound_rho}
  and \eqref{est1}, we have
  \begin{equation}\label{est4.2}
    \int_{\mathbb{R}_+}\psi^2\mathrm{d}x+\int_{\Omega'_1(s)}\theta\mathrm{d}x
    \lesssim \int_{\mathbb{R}_+}\rho\mathcal{E}\mathrm{d}x\lesssim 1.
  \end{equation}
  From Cauchy's inequality and \eqref{bound_rho}, we obtain
  \begin{equation}\label{J1}
    \begin{aligned}
   |\mathcal{J}_1|\leq~& \epsilon \int_0^t\int_{\mathbb{R}_+}\psi_x^2(\vartheta-2)_+
   +C(\epsilon)\int_0^t\int_{\mathbb{R}_+}\theta^2(\vartheta-2)_+\\
   \leq~&\epsilon \int_0^t\int_{\mathbb{R}_+}\psi_x^2(\vartheta-2)_+
   +C(\epsilon)\int_0^t\int_{\mathbb{R}_+}\theta(\vartheta-1)_+^2\\
    \leq~&\epsilon \int_0^t\int_{\mathbb{R}_+}\psi_x^2(\vartheta-2)_+
  +C(\epsilon)\int_0^t\sup_{\mathbb{R}_+}(\vartheta-1)_+^2,
  \end{aligned}
  \end{equation}
  and
  \begin{equation}\label{J2}
    \begin{aligned}
     |\mathcal{J}_2|\leq~&\epsilon\int_0^t\int_{\Omega'_2(s)}\vartheta_x^2
     +C(\epsilon)\int_0^t\int_{\Omega'_2(s)}\psi^2\theta^2\\
     \leq~&\epsilon\int_0^t\int_{\Omega'_2(s)}\vartheta_x^2
     +C(\epsilon)\int_0^t\int_{\Omega'_2(s)}\psi^2\left(\vartheta-1\right)_+^2\\
     \leq~&\epsilon\int_0^t\int_{\Omega'_2(s)}\vartheta_x^2
     +C(\epsilon)\int_0^t\sup_{\mathbb{R}_+}(\vartheta-1)_+^2.\\
   \end{aligned}
  \end{equation}
  Here we have used the fact that
  \begin{equation*}
    \theta\leq K(\vartheta-1)
  \end{equation*}
  with $K=2+\sup_{\mathbb{R}_+^2}\hat{\theta}$
  provided that $\vartheta\geq 2$.
  Let us define
  \begin{equation*}
    \mathcal{J}_3:=
    \int_0^t\int_{\Omega'_2(s)}\left[f_3(\vartheta-2)_+
      +c_v^{-1}f_3\psi^2+2f_2\psi(\vartheta-2)_+\right].
  \end{equation*}
  According to \eqref{f2} and \eqref{f3},
  we use \eqref{f_hat1} and \eqref{bound_rho} to deduce
  \begin{equation}\label{f_23}
      \begin{aligned}
     |(f_2,f_3)|
     \lesssim G
     +\big|(\Psi,\Psi_x)\big|\big|(\bar{U}_{x},\tilde{U}_{x},\tilde{U}_{xx})\big|.
  \end{aligned}
  \end{equation}
  with
  \begin{equation}\label{G1}
    G:=\big|\bar{U}_{xx}\big|+\big|\bar{U}_{x}\big|^2
    +|\bar{U}_x||\tilde{U}-U_m|
  +|\tilde{U}_x||\bar{U}-U_m|+
  |\bar{U}_x||\tilde{U}_x|.
  \end{equation}
  Hence, we have
  \begin{equation}\label{J3a}
    \mathcal{J}_3
    \lesssim \int_0^t\int_{\Omega'_2(s)}
     \left[G
     +\big|(\Psi,\Psi_x)\big|\big|(\bar{U}_{x},\tilde{U}_{x},\tilde{U}_{xx})\big|\right]
     \left[m_2^{-1}(\vartheta-2)_+^2+\psi^2\right].
  \end{equation}
  It follows from  Lemma \ref{lem_SRW} and \eqref{point1} that
  \begin{equation}\label{G2}
     \big\|G (s)\big\|_{L^1}
      \lesssim \bar{\delta}^{\frac{1}{3}} (1+s)^{-\frac{2}{3}},
  \end{equation}
  from which we get
 \begin{equation}\label{J3a1}
   \begin{aligned}
      &\int_0^t\int_{\Omega'_2(s)}
     G
  \left[m_2^{-1}(\vartheta-2)_+^2+\psi^2\right]\\
  &\quad \lesssim
    \bar{\delta}^{\frac{1}{3}}\int_0^t (1+s)^{-\frac{2}{3}}
   \|\psi \|\|\psi_x\|
      +
       m_2^{-1}\bar{\delta}^{\frac{1}{3}}
       \int_0^t\sup_{\mathbb{R}_+}(\vartheta-1)_+^2 \\
  &\quad \lesssim
  1+     N^2\bar{\delta}^{\frac{2}{3}}\int_0^t \|\psi_x\|^2
      +
       m_2^{-1}\bar{\delta}^{\frac{1}{3}}
       \int_0^t\sup_{\mathbb{R}_+}(\vartheta-1)_+^2.
   \end{aligned}
 \end{equation}
Next we have from \eqref{apriori} and Lemma \ref{lem_SRW} that
 \begin{equation}\label{J3a2}
   \begin{aligned}
      &m_2^{-1}\int_0^t\int_{\Omega'_2(s)}
     \big|(\Psi,\Psi_x)\big|\big|(\bar{U}_{x},\tilde{U}_{x},\tilde{U}_{xx})\big|
     (\vartheta-2)_+^2\\
  &\quad \lesssim
     m_2^{-1}\int_0^t\int_{\Omega'_2(s)}
    \big|(\bar{U}_{x},\tilde{U}_{x},\tilde{U}_{xx})\big|
    \left[N^2 (\vartheta-2)_+^2
      +\big|\Psi_x\big|^2\right] \\
  &\quad \lesssim
       m_2^{-1}N^2(\bar{\delta}+\tilde{\delta})
       \int_0^t\Big[ \sup_{\mathbb{R}_+}(\vartheta-1)_+^2
       +\|\Psi_x \|^2\Big].
   \end{aligned}
 \end{equation}
Since
  \begin{equation} \label{point3}
      \big|(\Psi,\Psi_x)\big| \psi^2
  \lesssim
      \left(1+\big|\Psi\big|\right)\big|\Psi\big|^3
      +\big|\Psi_x\big|^2
      \lesssim N\big|\Psi\big|^3+\big|\Psi_x\big|^2,
  \end{equation}
we have
  \begin{equation}\label{J3a3}
    \begin{aligned}
      & \int_0^t\int_{\Omega'_2(s)}
      \big|(\Psi,\Psi_x)\big|\big|(\bar{U}_{x},\tilde{U}_{x},\tilde{U}_{xx})\big|
      \psi^2
      \\   &\quad \lesssim
       \int_0^t\int_{\Omega'_2(s)}
       \left[N\big|\Psi\big|^3\big|\bar{U}_{x}\big|
       +N^2
       \big|\Psi\big|^2
       \big|(\tilde{U}_{x},\tilde{U}_{xx})\big|\right]
       +(\bar{\delta}+\tilde{\delta})
       \int_0^t\|\Psi_x \|^2
      \\   &\quad \lesssim
      1+
      \left[\bar{\delta}^{\frac{2}{9}} N^{\frac{26}{3}}
      +\tilde{\delta}N^2\right]
       \int_0^t \|\Psi_x \|^2
       +\tilde{\delta}m_1^{-1}
      N^4 \int_0^t
      \rho\Phi\left(\frac{\hat{\rho}}{\rho}\right)(s,0)\mathrm{d}s.
    \end{aligned}
  \end{equation}
To derive the last inequality in \eqref{J3a2},
we have used \eqref{R4b}, Young's inequality and
  \begin{equation*}
    \big|\bar{U}_{x}\big|\big|\Psi\big|^3
    \lesssim \|\Psi\|_{L^{\infty}}^{\frac{3}{2}}
      \|\bar{U}_{x}\|_{L^{4}}
      \|\Psi\|^{\frac{3}{2}}
      \lesssim \bar{\delta}^{\frac{1}{12}}N^{\frac{9}{4}}
       (1+s)^{-\frac{11}{16}}\|\Psi_x\|^{\frac{3}{4}}.
  \end{equation*}
  Plugging  \eqref{J3a1}, \eqref{J3a2}  and
  \eqref{J3a3} into \eqref{J3a},
   we deduce from  \eqref{asmp1}-\eqref{est2} that
  \begin{equation}\label{J3b}
    \begin{aligned}
      |\mathcal{J}_3|
  \lesssim  1+\int_0^t\sup_{\mathbb{R}_+}(\vartheta-1)_+^2.
    \end{aligned}
  \end{equation}
  Let us now consider the term
  \begin{equation*}
    \mathcal{J}_4:=
    2R\int_0^t\int_{\mathbb{R}_+}\psi(\hat{\rho}\hat{\theta})_x(\vartheta-2)_+,
  \end{equation*}
  which combined with \eqref{apriori} and \eqref{asmp1} yields
  \begin{equation}\label{J4}
    |\mathcal{J}_4|\lesssim\int_0^t\sup_{\mathbb{R}_+}(\vartheta-2)_+
    \|\psi\|\|(\hat{\rho}_x,\hat{\theta}_x)\|
    \lesssim\int_0^t\sup_{\mathbb{R}_+}(\vartheta-1)_+^2.
  \end{equation}
  For the term
  \begin{equation*}
    \mathcal{J}_5:=
    \int_0^t\int_{\Omega'_2(s)}\left[\mu\psi_x^2c_v^{-1}\psi^2
     -P\psi_xc_v^{-1}\psi^2-2\mu\psi\psi_x\vartheta_x\right],
  \end{equation*}
  we apply Cauchy's inequality and \eqref{est4.2} to deduce
  \begin{equation}\label{J5}
  \begin{aligned}
    |\mathcal{J}_5|\lesssim~& \epsilon \int_0^t\int_{\Omega'_2(s)}\vartheta_x^2
   +C(\epsilon)\int_0^t\int_{\Omega'_2(s)}\psi^2\psi_x^2
   +\int_0^t\int_{\Omega'_2(s)}\psi^2\theta^2\\
    \lesssim~&\epsilon \int_0^t\int_{\Omega'_2(s)}\vartheta_x^2
   +C(\epsilon)\int_0^t\int_{\Omega'_2(s)}\psi^2\psi_x^2
   +\int_0^t\int_{\Omega'_2(s)}\psi^2(\vartheta-1)_+^2\\
   \lesssim~&\epsilon \int_0^t\int_{\Omega'_2(s)}\vartheta_x^2
   +C(\epsilon)\int_0^t\int_{\Omega'_2(s)}\psi^2\psi_x^2
   +\int_0^t\sup_{\mathbb{R}_+}(\vartheta-1)_+^2.
  \end{aligned}
  \end{equation}
  We finally consider
  \begin{equation*}
    \mathcal{J}_6:=
    \int_0^t\int_{\Omega'_2(s)}c_v^{-1}\kappa\psi^2\vartheta_{xx}.
  \end{equation*}
  In order to estimate $\mathcal{J}_6$,
  we apply Lebesgue's dominated convergence theorem to find
  \begin{equation}\label{J6}
    \begin{aligned}
    \mathcal{J}_6
    =~&\frac{\kappa}{c_v}\lim_{\nu\to 0^+}\int_0^t\int_{\mathbb{R}_+}
    \varphi_{\nu}(\vartheta)\psi^2\vartheta_{xx}\\
    =~&\frac{\kappa}{c_v}\lim_{\nu\to 0^+}\int_0^t\int_{\mathbb{R}_+}
    \left[-2\varphi_{\nu}(\vartheta)\psi\psi_x\vartheta_{x}
    -\varphi'_{\nu}(\vartheta)\psi^2\vartheta_{x}^2\right]\\
    \leq ~&-\frac{\kappa}{c_v}\lim_{\nu\to 0^+}\int_0^t\int_{\mathbb{R}_+}
    2\varphi_{\nu}(\vartheta)\psi\psi_x\vartheta_{x}\\
    \leq ~&\epsilon \int_0^t\int_{\mathbb{R}_+}\vartheta_x^2
    +C(\epsilon)\int_0^T\int_{\mathbb{R}_+}\psi^2\psi_x^2.
  \end{aligned}
  \end{equation}
  where the approximate scheme $\varphi_{\nu}(\vartheta)$ is defined by
  \begin{equation*}
  \varphi_{\nu}(\vartheta)=
  \begin{cases}
    1, \quad \quad &\vartheta-2>\nu,\\
    (\vartheta-2)/\nu, \quad  \quad& 0<\vartheta-2\leq\nu,\\
    0, \quad \quad&\vartheta-2\leq 0.\\
  \end{cases}
\end{equation*}
Plugging \eqref{J1}-\eqref{J2}, \eqref{J3b}-\eqref{J6} into \eqref{id4.3}, we get
from \eqref{bound_rho} that
\begin{equation}\label{id4.3a}
  \begin{aligned}
    &\int_{\mathbb{R}_+}(\vartheta-2)_+^2\mathrm{d}x
      +\int_0^t\int_{\Omega'_2(s)}\left[\vartheta_x^2
      +\psi_x^2(\vartheta-2)_+\right]\\
      &\quad\lesssim 1+\epsilon\int_0^t\int_{\mathbb{R}_+}\vartheta_x^2
      +C(\epsilon)\int_0^t\int_{\mathbb{R}_+}\psi^2\psi_x^2
      +C(\epsilon)\int_0^t\sup_{\mathbb{R}_+}(\vartheta-1)_+^2.
  \end{aligned}
\end{equation}

\noindent{\bf Step 3}.
We obtain from \eqref{est1} that
\begin{equation}\label{est4.3}
  \begin{aligned}
 \int_0^t\int_{\mathbb{R}_+}\left[\vartheta_x^2+\psi_x^2\theta\right]
 \lesssim~&\int_0^t\int_{\Omega'_3(s)}
 \left[\vartheta_x^2+\psi_x^2(\vartheta-2)_+\right]
 +\int_0^t\int_{\mathbb{R}_+\setminus\Omega'_3(s)}
 \left[\frac{\vartheta_x^2}{\theta^2}+\frac{\psi_x^2}{\theta}\right]\\
 \lesssim~&\int_0^t\int_{\Omega'_2(s)}
 \left[\vartheta_x^2+\psi_x^2(\vartheta-2)_+\right]
 +1.
\end{aligned}
\end{equation}
Combining \eqref{est4.3} and \eqref{id4.3a}, and choosing $\epsilon$
sufficiently small, we have
\begin{equation}\label{id4.3b}
      \int_{\mathbb{R}_+}(\vartheta-2)_+^2\mathrm{d}x
      +\int_0^t\int_{\mathbb{R}_+}\left[\vartheta_x^2
      +\psi_x^2\theta\right]
      \lesssim 1
      +\int_0^t\sup_{\mathbb{R}_+}(\vartheta-1)_+^2
      +\int_0^t\int_{\mathbb{R}_+}\psi^2\psi_x^2.
\end{equation}

\noindent{\bf Step 4}.
To estimate the last term of \eqref{id4.3b}, we
multiply $\eqref{per1}_2$ by $\psi^3$ and then integrate the resulting identity
over $(0,t)\times\mathbb{R}_+$ to have
\begin{equation}\label{id4.4a}
 \begin{aligned}
   &\frac{1}{4}\int_{\mathbb{R}_+}\rho\psi^4\mathrm{d}x
   +3\mu\int_0^t\int_{\mathbb{R}_+}\psi^2\psi_x^2
   -\frac{1}{4}\int_{\mathbb{R}_+}\rho_0\psi_0^4\mathrm{d}x\\
   =~&
   3R\int_0^t\int_{\mathbb{R}_+}\psi^2\psi_x\hat{\theta}\phi
   +3R\int_0^t\int_{\mathbb{R}_+}\psi^2\psi_x\rho\vartheta
   +\int_0^t\int_{\mathbb{R}_+}f_2\psi^3.
 \end{aligned}
\end{equation}
From \eqref{est4.1} and \eqref{est4.2}, we have
\begin{equation}\label{id4.4a1}
 \begin{aligned}
    &\int_0^t\int_{\mathbb{R}_+}\psi^2\psi_x\hat{\theta}\phi
+\int_0^t\int_{\mathbb{R}_+\setminus\Omega'_2(s)}\psi^2\psi_x\rho\vartheta\\
    &\quad\lesssim\int_0^t\|\psi\|_{L_x^\infty}^2\|\psi_x\|
     \left[\|\phi\|+\|\vartheta\|_{L^2(\mathbb{R}_+\setminus\Omega_2'(s))}\right]
    \lesssim\int_0^t\|\psi_x\|^2,
  \end{aligned}
\end{equation}
We then apply Cauchy's inequality to derive
\begin{equation}\label{id4.4a2}
  \begin{aligned}
  \int_0^t\int_{\Omega_2'(s)}\psi^2\psi_x\rho\vartheta
  \leq~&
  \epsilon\int_0^t\int_{\mathbb{R}_+}\psi^2\psi_x^2
  +C(\epsilon)\int_0^t\int_{\Omega_2'(s)}\psi^2\vartheta^2\\
  \leq~& \epsilon\int_0^t\int_{\mathbb{R}_+}\psi^2\psi_x^2
  +C(\epsilon)\int_0^t\sup_{\mathbb{R}_+}\left(\vartheta-1\right)_+^2.
\end{aligned}
\end{equation}
In view of \eqref{f_23}, we utilize
\eqref{G2}, \eqref{point3}, \eqref{R4b} and the Young's inequalities to have
\begin{equation}\label{id4.4a3}
  \begin{aligned}
  \int_0^t\int_{\mathbb{R}_+}f_2\psi^3
    \lesssim ~&N\int_0^t\int_{\mathbb{R}_+}
    \left[G
     +\big|(\Psi,\Psi_x)\big|\big|(\bar{U}_{x},\tilde{U}_{x},\tilde{U}_{xx})\big|\right]
     \psi^2\\
    \lesssim~& 1+
      \left[\bar{\delta}^{\frac{2}{9}} N^{\frac{34}{3}}
      +\tilde{\delta}N^3\right]
       \int_0^t \|\Psi_x \|^2
       +\tilde{\delta}m_1^{-1}
      N^5 \int_0^t
      \rho\Phi\left(\frac{\hat{\rho}}{\rho}\right)(s,0)
      \mathrm{d}s.
\end{aligned}
\end{equation}
Plugging \eqref{id4.4a1} -\eqref{id4.4a3} into \eqref{id4.4a}, and taking $\epsilon$
sufficiently small, we derive from \eqref{asmp1}-\eqref{est2} that
\begin{equation}\label{id4.4b}
  \begin{aligned}
   \int_{\mathbb{R}_+}\psi^4\mathrm{d}x+\int_0^t\int_{\mathbb{R}_+}\psi^2\psi_x^2
   \lesssim1+\int_0^t\int_{\mathbb{R}_+}\psi_x^2
   +\int_0^t\sup_{\mathbb{R}_+}\left(\vartheta-1\right)_+^2.
 \end{aligned}
\end{equation}
It follows from \eqref{est1} that
\begin{equation}\label{id4.4b1}
  \int_0^t\int_{\mathbb{R}_+}\psi_x^2
 \leq \epsilon \int_0^t\int_{\mathbb{R}_+}\theta\psi_x^2
 +C(\epsilon)\int_0^t\int_{\mathbb{R}_+}\frac{\psi_x^2}{\theta}
   \leq \epsilon \int_0^t\int_{\mathbb{R}_+}\theta\psi_x^2
   +C(\epsilon).
\end{equation}
Combination of  \eqref{id4.4b1} and \eqref{id4.4b} yields
\begin{equation}\label{id4.4c}
  \int_{\mathbb{R}_+}\psi^4\mathrm{d}x+\int_0^t\int_{\mathbb{R}_+}(1+\psi^2)\psi_x^2
  \lesssim C(\epsilon)+\epsilon \int_0^t\int_{\mathbb{R}_+}\theta\psi_x^2
  +\int_0^t\sup_{\mathbb{R}_+}\left(\vartheta-1\right)_+^2.
\end{equation}
We plug \eqref{id4.4c} into \eqref{id4.3b} and choose $\epsilon$
suitable small to find
\begin{equation}\label{est4.4}
    \int_{\mathbb{R}_+}
    \left[\left(\vartheta-2\right)_+^2+\psi^4\right]\mathrm{d}x
    +\int_0^t\int_{\mathbb{R}_+}\left[\vartheta_x^2+\psi_x^2\left(1+\theta+\psi^2\right)\right]
    \lesssim 1+\int_0^t\sup_{\mathbb{R}_+}\left(\vartheta-1\right)_+^2.
 \end{equation}

\noindent{\bf Step 5}.
It remains to estimate the last term of \eqref{est4.4}.
According to the fundamental theorem of calculus, we have from \eqref{est4.2} that
\begin{equation}\label{est4.5}
  \begin{aligned}
    \int_0^T\sup_{\mathbb{R}_+}\left(\vartheta-1\right)_+^2
%    =~&\int_0^T\sup_{\mathbb{R}_+}\left[\int_x^{\infty}
%    \partial_x\left(\vartheta-1\right)_+\right]^2\\
    \leq ~&\int_0^T\left[\int_{\Omega_{1}'(s)}\left|\vartheta_x\right|\right]^2\\
    \leq ~&\int_0^T\left[\int_{\Omega_{1}'(s)}\frac{\vartheta_x^2}{\theta}
    \int_{\Omega_{1}'(s)}\theta\right]\\
%    \leq ~&C\int_0^T\int_{\Omega_{1}'(s)}\frac{\vartheta_x^2}{\theta}\\
    \leq ~&\epsilon\int_0^T\int_{\mathbb{R}_+}\vartheta_x^2
    +C(\epsilon)\int_0^T\int_{\mathbb{R}_+}\frac{\vartheta_x^2}{\theta^2}\\
    \leq ~&\epsilon\int_0^T\int_{\mathbb{R}_+}\vartheta_x^2+C(\epsilon).
  \end{aligned}
\end{equation}
Plug \eqref{est4.5} into \eqref{est4.4} and choose $\epsilon>0$
suitable small to obtain \eqref{est4}.
This completes the proof of the lemma.
\end{proof}
We obtain the uniform bound of the $H^1_x$-norm of
$(\phi,\psi,\vartheta)(t,x)$ uniformly in  time $t$ in the next lemma.
\begin{lemma}\label{lem_the2}
Suppose that the conditions listed in Lemma \ref{lem_bas} hold.  Then
  \begin{equation}
   \label{est5}
    \sup_{0\leq t\leq T}\|(\phi,\psi,\vartheta)(t)\|_1^2
    +\int_0^T
    \left[\|\sqrt\theta\phi_x(t)\|^2+\|(\psi_{x},\vartheta_{x})(t)\|_1^2\right]
    \mathrm{d}t
    \leq C_2^2,
  \end{equation}
    where  the positive constant $C_2$ depends only
on $\inf_{x\in\mathbb{R}_+}\{\rho_0(x),\theta_0(x)\}$
  and $\|(\phi_0,\psi_0,\vartheta_0)\|_1$.
\end{lemma}
\begin{proof}
  First, plugging \eqref{bound_rho} into \eqref{id4e},
  we deduce
  \begin{equation}\label{est5.1}
    \int_{\mathbb{R}_+}\phi_x^2+\int_0^t\int_{\mathbb{R}_+}\theta\phi_x^2
    \lesssim 1+\int_0^t\int_{\mathbb{R}_+}\left[\psi_x^2+\vartheta_x^2
    +\frac{\vartheta_x^2}{\theta^2}\right]\lesssim 1,
  \end{equation}
where we employed \eqref{est1} and \eqref{est4} in the last inequality.

  Next, multiply $\eqref{per1}_2$ by ${\psi_{xx}}/{\rho}$ to derive
  \begin{equation*}
    \left(\frac{\psi_x^2}{2}\right)_t-\left[\psi_t\psi_x+\frac12u\psi_x^2\right]_x
    +\frac12u_x\psi_x^2+\mu\frac{\psi_{xx}^2}{\rho}
    =\frac{(P-\hat{P})_x}{\rho}\psi_{xx}-\frac{f_2}{\rho}\psi_{xx}.
  \end{equation*}
  Integrating this last identity over $(0,t)\times\mathbb{R}_+$,
   we obtain from \eqref{bound_rho}
  and Cauchy's inequality that
  \begin{equation}\label{est5.2a}
  \begin{aligned}
    &\int_{\mathbb{R}_+}\psi_x^2\mathrm{d}x
    -\int_0^t\psi_x^2(s,0)\mathrm{d}s +\int_0^t\int_{\mathbb{R}_+}\psi_{xx}^2
    \\&\quad
    \lesssim 1
    +\int_0^t\int_{\mathbb{R}_+}\left[(P-\hat{P})_x^2
    +f_2^2+|\hat{u}_x|\psi_x^2+|\psi_x|^3\right].
  \end{aligned}
  \end{equation}
  Apply Sobolev's inequality and \eqref{est4} to obtain
  \begin{equation}\label{est5.2a1}
   \begin{aligned}
     \int_0^t\psi_x^2(s,0)\mathrm{d}s+\int_0^t\int_{\mathbb{R}_+}|\psi_x|^3
     \lesssim~&\int_0^t\|\psi_x\|\|\psi_{xx}\|
     +\int_0^t\|\psi_x\|^{\frac52}\|\psi_{xx}\|^{\frac12}\\
     \lesssim~&\epsilon\int_0^t\int_{\mathbb{R}_+}\psi_{xx}^2
     +C(\epsilon)\int_0^t\left[\|\psi_x\|^2
     +\|\psi_x\|^{\frac{10}{3}}\right]\\
     \lesssim ~&C(\epsilon)\left[1+
     \sup_{0\leq s\leq t}\|\psi_x(s)\|^{\frac{4}{3}}\right]
     +\epsilon\int_0^t\int_{\mathbb{R}_+}\psi_{xx}^2.
   \end{aligned}
  \end{equation}
  Using
  $
    (P-\hat{P})_x=R(\theta\phi_x+\rho\vartheta_x
    +\phi\hat{\theta}_x+\vartheta\hat{\rho}_x),
  $
  we derive from
  \eqref{f_23}, \eqref{est5.1}, \eqref{est4}
  that
  \begin{equation}\label{est5.2a2}
    \begin{aligned}
     \int_0^T\int_{\mathbb{R}_+}\left[(P-\hat{P})_x^2
    +|\hat{u}_x|\psi_x^2\right]
    \lesssim ~&\int_0^T\int_{\mathbb{R}_+}\left[|(\theta\phi_x,\psi_x,\vartheta_x)|^2
    +|(\hat{\rho}_x,\hat{u}_x,\hat{\theta}_x)|^2|\Psi|^2
    \right]\\
     \lesssim~& 1+\|\theta\|_{L^\infty([0,T]\times\mathbb{R}_+)}.
   \end{aligned}
  \end{equation}
  According to \eqref{f_23} and \eqref{G1}, we have from
  Lemma \ref{lem_SRW}, \eqref{R4b}, \eqref{asmp1}, \eqref{est1} and
  \eqref{est5.1} that
  \begin{equation}\label{f_23a}
    \int_0^T\int_{\mathbb{R}_+}|(f_2,f_3)|^2\lesssim 1.
  \end{equation}
  We plug \eqref{est5.2a1}-\eqref{f_23a} into \eqref{est5.2a} to get
  \begin{equation*}
    \sup_{0\leq t\leq T}\|\psi_x(t)\|^2+\int_0^T\int_{\mathbb{R}_+}\psi_{xx}^2
    \lesssim 1+\|\theta\|_{L^\infty([0,T]\times\mathbb{R}_+)}
    +\sup_{0\leq t\leq T}\|\psi_x(t)\|^{\frac43}.
  \end{equation*}
  Then Young's inequality yields the estimate
  \begin{equation}\label{est5.2b}
    \sup_{0\leq t\leq T}\|\psi_x(t)\|^2+\int_0^T\int_{\mathbb{R}_+}\psi_{xx}^2
    \lesssim 1+\|\theta\|_{L^\infty([0,T]\times\mathbb{R}_+)}.
  \end{equation}

  Next, multiply $\eqref{per1}_3$ by $ {\vartheta_{xx}}/{\rho}$ and integrate
  the resulting identity over $(0,T)\times\mathbb{R}_+$ to have
  \begin{equation*}
    \begin{aligned}
      \frac{c_v}{2}\int_{\mathbb{R}_+}\vartheta_x^2\mathrm{d}x
      +\kappa\int_0^T\int_{\mathbb{R}_+}\frac{\vartheta_{xx}^2}{\rho}
      =\frac{c_v}{2}\int_{\mathbb{R}_+}\vartheta_{0x}^2\mathrm{d}x
      +\int_0^T\int_{\mathbb{R}_+}\left[c_v u\vartheta_x
      -\mu\frac{\psi_x^2}{\rho} + R\theta\psi_x
      -\frac{f_3}{\rho}\right]\vartheta_{xx},
    \end{aligned}
  \end{equation*}
  which combined with \eqref{bound_rho} implies
  \begin{equation}\label{est5.3a}
    \begin{aligned}
      \int_{\mathbb{R}_+}\vartheta_x^2\mathrm{d}x
      +\int_0^T\int_{\mathbb{R}_+}\vartheta_{xx}^2
      \lesssim~&1+\int_0^T\int_{\mathbb{R}_+}
      \left[u^2\vartheta_x^2+\|\psi_x\|_{L^\infty}^2\psi_x^2+\theta^2\psi_x^2+h^2\right]\\
      \lesssim~&1+\int_0^T(1+\|\psi\|\|\psi_x\|)\|\vartheta_x\|^2
      +\int_0^T\|\psi_x\|^3\|\psi_{xx}\|\\
      &+\|\theta\|_{L^\infty([0,T]\times\mathbb{R}_+)}\int_0^T
      \int_{\mathbb{R}_+}\theta\psi_x^2
      +\int_0^T\int_{\mathbb{R}_+}f_3^2.
    \end{aligned}
  \end{equation}
  From \eqref{est5.2b}, we have
  \begin{align*}
    \int_0^T\|\psi_x\|^3\|\psi_{xx}\|
   \lesssim\sup_{0\leq t\leq T}\|\psi_x\|^2\int_0^T\left(\|\psi_x\|^2+\|\psi_{xx}\|^2\right)
   \lesssim 1+\|\theta\|_{L^\infty([0,T]\times\mathbb{R}_+)}^2.
  \end{align*}
 In light of \eqref{est5.1}, \eqref{est5.2b} and \eqref{est4},
we then obtain
 \begin{equation}\label{est5.3b}
   \sup_{0\leq t\leq T}\int_{\mathbb{R}_+}\vartheta_x^2\mathrm{d}x
      +\int_0^T\int_{\mathbb{R}_+}\vartheta_{xx}^2
   \lesssim 1+\|\theta\|_{L^\infty([0,T]\times\mathbb{R}_+)}^2.
 \end{equation}

 Finally, it follows from \eqref{est4} and \eqref{est5.3b} that
  \begin{equation*}
  \begin{aligned}
   \|\theta-\hat{\theta}\|_{L^\infty([0,T]\times\mathbb{R}_+)}^2
    \lesssim \sup_{0\leq t\leq T}\|\vartheta(t)\|\|\vartheta_x(t)\|
    \lesssim1+\|\theta\|_{L^\infty([0,T]\times\mathbb{R}_+)}.
  \end{aligned}
  \end{equation*}
  from which we have
  \begin{equation}\label{upper2}
   \theta(t,x)\lesssim 1 \quad {\rm for\ all}\ (t,x)\in[0,T]\times\mathbb{R}_+.
  \end{equation}
  Combine
  \eqref{est5.2a}, \eqref{est5.2b} and \eqref{est5.3b} to give
  \begin{equation*}
        \sup_{0\leq t\leq T}\int_{\mathbb{R}_+}[\phi_x^2+\psi_x^2+\vartheta_x^2]\mathrm{d}x
    +\int_0^T\int_{\mathbb{R}_+}[\theta\phi_x^2+\psi_{xx}^2+\vartheta_{xx}^2]
    \lesssim 1,
  \end{equation*}
  which together with \eqref{est4} yields \eqref{est5}.
This completes the proof of the lemma.
\end{proof}

\subsection{Local lower bound of temperature}\label{sec_tem}
In this subsection, we employ the maximum principle
to get the lower bound for the temperature, which does depend on the time $t$.
\begin{lemma}
  \label{lem_lower2}
Suppose that the conditions listed in Lemma \ref{lem_bas} hold. Then
  \begin{equation} \label{lower2}
    \inf_{\mathbb{R}_+}\theta(t,\cdot)
    \geq\frac{\inf_{\mathbb{R}_+}\theta(s,\cdot)}{C_3\inf_{\mathbb{R}_+}
    \theta(s,\cdot)(t-s)+1}
    \quad
    {\rm for}\ 0\leq s\leq t\leq T,
  \end{equation}
  where  the positive constant $C_3$ depends solely
on $\inf_{x\in\mathbb{R}_+}\{\rho_0(x),\theta_0(x)\}$
  and $\|(\phi_0,\psi_0,\vartheta_0)\|_1$.
\end{lemma}
\begin{proof}
  It follows from $\eqref{NS_E}_3$ that $\theta$ satisfies
  \begin{equation*}
    \theta_t+u\theta_x-\frac{\kappa}{c_v\rho}\theta_{xx}
    =\frac{\mu}{c_v\rho}\left[u_x^2-\frac{P}{\mu}u_x\right]
    \geq-\frac{P^2}{4\mu c_v\rho}=-\frac{R^2\rho}{4\mu c_v}\theta^2.
  \end{equation*}
  Hence we deduce from \eqref{bound_rho} that
  \begin{equation*}
    \theta_t+u\theta_x-\frac{\kappa}{c_v\rho}\theta_{xx}
    +C_3\theta^2\geq 0.
  \end{equation*}
  Let $\Theta:=\theta-\underline{\theta}$ with
  $\underline{\theta}:=\frac{\inf_{\mathbb{R}_+}\theta(s,\cdot)}
  {C_3\inf_{\mathbb{R}_+}\theta(s,\cdot)(t-s)+1}$. We observe
  \begin{equation*}
    \Theta|_{x=0,\infty}\geq 0,\quad
    \Theta|_{t=s}\geq 0,
  \end{equation*}
  and
  \begin{equation*}
    \begin{aligned}
      \Theta_t+u\Theta_x-\frac{\kappa}{c_v\rho}\Theta_{xx}+
      C_2(\theta+\underline{\theta})\Theta
      =\theta_t+u\theta_x
      -\frac{\kappa}{c_v\rho}\theta_{xx}+C_2\theta^2
       \geq 0.
    \end{aligned}
  \end{equation*}
  Applying the weak maximum principle
  (see \cite[Section 7.1]{E10MR2597943}),
we have that $\Theta(t,x)\geq 0$
  for $0\leq s\leq t\leq T$ and $x\in\mathbb{R}_+$.
This completes the proof of the lemma.
\end{proof}
\subsection{Proof of Theorem \ref{thm2}}\label{sec_proof}
This subsection is devoted to proving the stability of
the superposition of a rarefaction wave and a non-degenerate stationary
solution, i.e. Theorem \ref{thm2}.
To this end,
we first give the local existence of solutions
to the problem \eqref{per1}-\eqref{per2} in the following proposition.
It can be proved by the standard iteration method (see \cite{HMS04MR2040072}
for example)
and hence we omit the proof  for brevity.
\begin{proposition}[Local existence] \label{Pro_loc}
   Suppose that the conditions in Theorem \ref{thm2} hold.
   Let $M$, $\lambda_1$
   and $\lambda_2$ be some  positive constants such that
   $\|(\phi_0,\psi_0,\vartheta_0)\|_1\leq M$,
   $\phi_0(x)+\hat{\rho}(x)\geq \lambda_1$ and
    $\vartheta_0(x)+\hat{\theta}(x)\geq \lambda_2$
    for all $x\in\mathbb{R}_+$.
Then  there exists a positive constant
$T_0=T_0(\lambda_{1},\lambda_{2},M)$,
depending only on $\lambda_{1}$, $\lambda_{2}$ and $M$, such that
the problem \eqref{per1}-\eqref{per2}
admits a unique solution $(\phi,\psi,\vartheta)\in
X\left(0,T_0;\frac{1}{2}\lambda_1,
\frac{1}{2}\lambda_2,2M\right)$.
\end{proposition}

Next we will give the proof of Theorem \ref{thm2} in six steps
by employing the continuation argument.

\vspace*{2mm}

\noindent{\bf Step 1}.
Let $\Pi$ and $\lambda_i $ $(i=1,2,3)$
be some positive constants
such that $\|(\phi_0,\psi_0,\vartheta_0)\|_1\leq \Pi$ and
\begin{equation*}\label{3.1}
   \rho_0(x)\geq \lambda_1,\quad
   \theta_0(x)\geq \lambda_2,\quad
   \hat{\theta}(t,x)\geq \lambda_3\quad
  {\rm for\ all}\ t,x\geq 0.
\end{equation*}
Set
$T_1=128\lambda_3^{-4}C_2^4$,
where $C_2$ is exactly the same constant as in  \eqref{est5}.
Applying Proposition \ref{Pro_loc}, we infer
that the problem \eqref{per1}-\eqref{per2} has
a unique solution $(\phi,\psi,\vartheta)\in X(0,t_1;\frac{1}{2}\lambda_1,
 \frac{1}{2}\lambda_2,2 \Pi)$
for some positive constant
$$ t_1=\min\{T_1,T_0(\lambda_1,\lambda_2,\Pi)\}.$$

Let $0<
\delta\leq \delta_1$ with
 $$\Xi\left(\tfrac{1}{2}{\lambda_1},
\tfrac{1}{2}{\lambda_2},2\Pi\right)\delta_1=\epsilon_0.$$
Then we can apply Lemmas
\ref{lem_boun},  \ref{lem_the2} and \ref{lem_lower2}
with $T=t_1$ to obtain that for each $t\in[0,t_1]$,
the local solution $(\phi,\psi,\vartheta)$ constructed above satisfies that
\begin{equation}\label{3.2}
      \theta(t,x)\geq \frac{\lambda_2}
      {C_3 \lambda_2T_1+1}=:C_4 \quad {\rm for\ all}\ x\in \mathbb{R}_+,
\end{equation}
and
\begin{equation}\label{3.3}
  \begin{gathered}
    C_1^{-1}\leq \rho(t,x)\leq C_1
    \quad {\rm for\ all}\ x\in \mathbb{R}_+,\\
    \|(\phi,\psi,\vartheta)(t)\|_1^2
    +\int_0^{t}
    \left[\|\sqrt\theta\phi_x(s)\|^2+\|(\psi_{x},\vartheta_{x})(s)\|_1^2\right]
    \mathrm{d}s
    \leq C_2^2.
  \end{gathered}
\end{equation}

\noindent{\bf Step 2}.
If we take $(\phi,\psi,\vartheta)(t_1,\cdot)$ as the initial data,
we can apply Proposition \ref{Pro_loc}
and extend the local solution  $(\phi,\psi,\vartheta)$ to the time
interval $[0,t_1+t_2]$ with
$$t_2=\min\{T_1-t_1, T_0(C_1^{-1},
C_4,C_2)\}.$$
Moreover,   for all $(t,x)\in[t_1,t_1+t_2]\times\mathbb{R}_+$,
\begin{equation*}\label{3.4}
   \rho(t,x)\geq \tfrac{1}{2}C_1^{-1},\quad \theta(t,x)\geq \tfrac{1}{2}C_4,
   \quad \|(\phi,\psi,\vartheta)(t)\|_1\leq 2C_2.
\end{equation*}
Take $0<\delta\leq \min\{\delta_1,\delta_2\}$ with
$$\Xi\left(\tfrac{1}{2}C_1^{-1},
\tfrac{1}{2}C_4,2C_2\right)\delta_2=\epsilon_0.$$
Then we can employ
Lemmas
\ref{lem_boun},  \ref{lem_the2} and \ref{lem_lower2}
with $T=t_1+t_2$ to deduce that
the local solution $(\phi,\psi,\vartheta)$ satisfies \eqref{3.2} and
\eqref{3.3} for each $t\in[0,t_1+t_2]$.

\noindent{\bf Step 3}.
We repeat the argument in Step 2, to extend our
solution $(\phi,\psi,\vartheta)$ to the time interval $[0,t_1+t_2+t_3]$
with
$$t_3=\min\{T_1-(t_1+t_2),T_0(C_1^{-1},
C_4,C_2)\}.$$
Assume that  $0<\delta\leq \min\{\delta_1,\delta_2\}$.
Continuing, after finitely many steps we construct the unique solution
$(\phi,\psi,\vartheta)$ existing on $[0,T_1]$ and satisfying
 \eqref{3.2} and \eqref{3.3} for each $t\in[0,T_1]$.

\noindent{\bf Step 4}.
Since $T_1\geq 128\lambda_3^{-4}C_3^4$ and
$$\sup_{0\leq t\leq T_1}\|\vartheta(t)\|_1^2
    +\int_{{T_1}/{2}}^{T_1}\|\vartheta_{x}(t)\|_1^2\mathrm{d}t
    \leq C_2^2,$$
we can find a $t_0'\in[T_1/2,T_1]$ such that
$$\|\vartheta(t_0')\|\leq C_2,\quad
\|\vartheta_x(t_0')\|\leq \tfrac{1}{8}{C_2^{-1}}\lambda_3^2.$$
Sobolev's inequality yields
$$
\|\vartheta(t_0')\|_{L^{\infty}}\leq
\sqrt2\|\vartheta(t_0')\|^{\frac12}\|\vartheta_x(t_0')\|^{\frac12}
\leq \tfrac{1}{2}\lambda_3,$$
and so
\begin{equation*}\label{3.5}
\theta(t_0',x)\geq \hat\theta(t_0',x)-\|\vartheta(t_0')\|_{L^{\infty}}
\geq \tfrac{1}{2}\lambda_3 \quad {\rm for\ all}\ x\in\mathbb{R}_+.
\end{equation*}
We note here that
\begin{equation*}\label{3.6}
\|(\phi,\psi,\vartheta)(t_0')\|_1
    \leq C_2,\quad
  \rho(t_0',x)\geq C_1^{-1} \quad {\rm for\ all}\ x\in \mathbb{R}_+.
\end{equation*}
Applying Proposition \ref{Pro_loc} again by
taking $(\phi,\psi,\vartheta)(t_0',\cdot)$ as the initial data,
we see that the problem \eqref{per1}-\eqref{per2} admits
a unique solution  $(\phi,\psi,\vartheta)\in X(t_0',t_0'+t_1';
\tfrac{1}{2}C_1^{-1},
\tfrac{1}{4}\lambda_3,2C_2)$ with
$$t_1'=\min\{T_1,T_0(C_1^{-1},
\tfrac{1}{2}\lambda_3,C_2)\}.$$

If we take $0<\delta\leq \min\{\delta_1,\delta_2,\delta_3\}$ with
$$\Xi\left(\tfrac{1}{2}C_1^{-1},
\tfrac{1}{4}\lambda_3,2C_2\right)\delta_3=\epsilon_0,$$
then we can apply
Lemmas
\ref{lem_boun},  \ref{lem_the2} and \ref{lem_lower2}
with $T=t_0'+t_1'$ to obtain
 that for each time $t\in[t_0',t_0'+t_1']$,
the local solution $(\phi,\psi,\vartheta)$ satisfies
\eqref{3.3} and
\begin{equation}\label{3.8}
      \theta(t,x)\geq \frac{\inf_{\mathbb{R}_+}\theta(t_0',\cdot)}
      {C_3 \inf_{\mathbb{R}_+}\theta(t_0',\cdot)T_1+1}
      \geq \frac{\lambda_3}
      {C_3 \lambda_3T_1+2}=:C_5
      \quad {\rm for\ all}\ x\in \mathbb{R}_+.
\end{equation}

\noindent{\bf Step 5}.
Next if we take $(\phi,\psi,\vartheta)(t_0'+t_1',\cdot)$ as the initial data,
we  apply Proposition \ref{Pro_loc}
 and construct the solution  $(\phi,\psi,\vartheta)$ existing on the time
interval $[0,t_0'+t_1'+t_2']$ with
$$t_2'=
\min\{T_1-t'_1,T_0(C_1^{-1},C_5,C_2)\}$$
       and satisfying
\begin{equation*}
   \rho(t,x)\geq \tfrac{1}{2}C_1^{-1},\quad
  \theta(t,x)\geq  \tfrac{1}{2}C_5,\quad
  \|(\phi,\psi,\vartheta)(t)\|_1\leq 2C_2
\end{equation*}
 for all $(t,x)\in[t_0'+t_1',t_0'+t_1'+t_2']\times\mathbb{R}_+.$
Let $0<\delta\leq \min\{\delta_1,\delta_2,\delta_3,\delta_4\}$ with
$$\Xi\left(\tfrac{1}{2}C_1^{-1},\tfrac{1}{2}C_5,2C_2\right)\delta_4=\epsilon_0.$$
Then we can infer from
Lemmas
\ref{lem_boun},  \ref{lem_the2} and \ref{lem_lower2}
with $T=t_0'+t_1'+t_2'$
that
the local solution $(\phi,\psi,\vartheta)$ satisfies \eqref{3.8} and
\eqref{3.3} for each $t\in[t_0',t_0'+t_1'+t_2']$.
By assuming $0<\delta\leq \min\{\delta_1,\delta_2,\delta_3,\delta_4\}$,
we can repeatedly apply the argument above to extend the local solution
to the time interval $[0,t_0'+T_1]$.
Furthermore, we deduce that
\eqref{3.8} and
\eqref{3.3} hold for each $t\in[t_0',t_0'+T_1]$.
In view of $t_0'+T_1\geq \frac{3}{2}T_1$, we have shown that
the problem \eqref{per1}-\eqref{per2} admits a unique solution $(\phi,\psi,\vartheta)$ on
$[0,\frac{3}{2}T_1]$.

\noindent{\bf Step 6}.
We take $0<\delta\leq \min\{\delta_1,\delta_2,\delta_3,\delta_4\}$.
As in Steps 4 and 5, we can find $t_0''\in[t_0'+T_1/2,t_0'+T_1]$ such that
the problem \eqref{per1}-\eqref{per2} admits a unique solution $(\phi,\psi,\vartheta)$ on
$[0,t_0''+T_1]$, which satisfies
\eqref{3.8} and
\eqref{3.3} for each $t\in[t_0',t_0''+T_1]$.
Since $t_0''+T_1\geq t_0'+\frac{3}{2}T_1\geq 2T_1$,
we have extended the local solution $(\phi,\psi,\vartheta)$
to $[0,2T_1]$.
Repeating the above procedure, we can then extend the
solution $(\phi,\psi,\vartheta)$ step by step to a global one provided that
$\delta\leq \min\{\delta_1,\delta_2,\delta_3,\delta_4\}$.
Choosing $\epsilon_2=\min\{\delta_1,\delta_2,\delta_3,\delta_4\}$, we
derive that
the problem \eqref{per1} has a unique solution
$(\phi,\psi,\vartheta)\in X(0,\infty;C_1^{-1},\min\{C_4,C_5\},C_2)$ satisfying
\eqref{3.3} for each $t\in[0,\infty)$.

Therefore, we can find constant $C_6$ depending only on
$\inf_{x\in\mathbb{R}_+}\{\rho_0(x),\theta_0(x)\}$
  and $\|(\phi_0,\psi_0,\vartheta_0)\|_1$ such that
\begin{equation*}
    \sup_{0\leq t<\infty}\|(\phi,\psi,\vartheta)(t)\|_1^2
    +\int_0^{\infty}
    \left[\|\phi_x(t)\|^2+\|(\psi_{x},\vartheta_{x})(t)\|_1^2\right]
    \mathrm{d}t
    \leq C_6^2,
\end{equation*}
from which the large-time behavior \eqref{stability2} follows in a standard argument
(cf. \cite{MN04book}).
This completes the proof of Theorem \ref{thm2}.

%\begin{center}
%{\bf Acknowledgement}
%\end{center}

\bibliographystyle{siam}

\bibliography{outflow}

\end{document}